\documentclass[11pt]{article}

\usepackage{latexsym}
\usepackage{amsthm}
\usepackage{amsmath}
\usepackage{amsfonts, epsfig, amsmath, amssymb, color, amscd}
\usepackage[cmtip,arrow]{xy}
\usepackage{pb-diagram, pb-xy}
\usepackage{amssymb,epsfig}

\newfont{\sdbl}{msbm9}
\newfont{\dbl}{msbm10 at 12pt}
\theoremstyle{definition}
\newcommand{\da}{{\mbox{\dbl A}}}
\newcommand{\dpp}{{\mbox{\dbl P}}}
\newcommand{\dz}{{\mbox{\dbl Z}}}

\newcommand{\dn}{{\mbox{\dbl N}}}

\newcommand{\sdz}{{\mbox{\sdbl Z}}}
\newcommand{\sdn}{{\mbox{\sdbl N}}}
\newcommand{\dc}{{\mbox{\dbl C}}}
\newcommand{\dq}{{\mbox{\dbl Q}}}

\newcommand{\ord}{\mathop{\rm ord}\nolimits}

\newcommand{\gr}{\mathop {\rm gr}}

\newcommand{\Der}{\mathop {\rm Der}}
\newcommand{\End}{\mathop {\rm End}}
\newcommand{\Adm}{\mathop {\rm Adm}}
\newcommand{\Ker}{\mathop {\rm Ker}}
\newcommand{\HT}{\mathop {\rm HT}}
\newcommand{\LT}{\mathop {\rm LT}}
\newcommand{\Ch}{\mathop {\rm Ch}}
\newcommand{\Sup}{\mathop {\rm Supp}}
\newcommand{\Quot}{\mathop {\rm Quot}}

\newcommand{\trdeg}{\mathop {\rm trdeg}}

\newcommand{\Spec}{\mathop {\rm Spec}}
\newcommand{\Proj}{\mathop {\rm Proj}}

\makeatother
\newtheorem{defin}{Definition}[section]
\newtheorem{nt}{Remark}[section]
\newtheorem{ex}{Example}[section]

\theoremstyle{plain}
\newtheorem{prop}{Proposition}[section]
\newtheorem{theo}{Theorem}[section]
\newtheorem{lemma}{Lemma}[section]
\newtheorem{corol}{Corollary}[section]

\newcommand{\eqdef}{\stackrel{\rm def}{=}}

\newcommand{\Ord}{\mathop {\rm \bf ord}}
\newcommand{\Proof}{{\noindent \bf Proof}}
\newcommand{\co}{{{\cal O}}}
\newcommand{\cf}{{{\cal F}}}
\newcommand{\cq}{{{\cal Q}}}
\newcommand{\cs}{{{\cal S}}}
\newcommand{\cm}{{{\cal M}}}
\newcommand{\ci}{{{\cal I}}}
\newcommand\limproj{\mathop{\underleftarrow{\lim}}}
\newcommand\limind{\mathop{\underrightarrow{\lim}}}

\title{On rings of commuting partial differential operators. \footnote{the  author is supported by RFBR grant no. 11-01-00145-a, by grant SSh ¹ 1410.2012.1, 
by grant of National Scientific Projects no. 14.740.11.0794, and by the Government of the Russian Federation for support of research projects implemented by leading scientists at Lomonosov Moscow State University under the agreement  No. 11.G34.31.0054.
}}
\author{A.B.Zheglov}
\date{}

\begin{document}

\maketitle

\begin{abstract} We give a natural generalization of the classification of commutative rings of ordinary differential operators, given in works of Krichever, Mumford, Mulase, and determine commutative rings of operators in a completed ring of partial differential operators in two variables (satisfying certain mild conditions)  in terms of Parshin's generalized geometric data. It uses a generalization of M.Sato's theory and is constructible in both ways.





\end{abstract}

\section{Introduction}

The problem of classification of commutative rings of ordinary differential operators was inspired already by works of Wallenberg  \cite{W} and Schur \cite{S}, and then has been studied by many authors and in diverse context of motivations, including Burchnall-Chaundy \cite{BC}, Gelfand-Dikii \cite{GD}, Krichever \cite{Kr}, Drinfeld \cite{Dr}, Mumford \cite{Mum1}, Segal-Wilson \cite{SW}, Verdier \cite{V} and Mulase \cite{Mu}.

Recall that the commutative algebras of ordinary differential operators correspond to spectral data. Thus, if we have a ring of commuting operators generated over a ground field $k$ by two ordinary differential operators 
$$
P_1=\partial_x^n+u_{n-1}(x)\partial_x^{n-1}+\ldots +u_0(x), \mbox{\quad} 
P_2=\partial_x^m+v_{m-1}(x)\partial_x^{m-1}+\ldots +v_0(x), 
$$
then, as it was found already by Burchnall-Chaundy \cite{BC}, there is a non zero polynomial $Q(\lambda ,\mu )$ such that $Q(P_1,P_2)=0$. A completion $C$ of the curve $Q(\lambda ,\mu )=0$ is called a {\it spectral curve}. At a generic point $(\lambda ,\mu )$ the space of eigenfunction $\psi$ (Baker-Akhieser functions): 
$$
P_1\psi =\lambda \psi , \mbox{\quad} P_2\psi =\mu \psi 
$$ 
has dimension $r$, and these functions are sections of a torsion free sheaf $\cf$ of rank $r$ on the spectral curve (for more precise statements and details see works cited above). The completion of the curve $Q(\lambda ,\mu )=0$ is obtained by adding a smooth point $P$ (this is not necessary the projective closure in $\dpp^2$!), and the triple $(C,P,\cf )$ is a part of the so called {\it spectral data}. 

Generalizing this result of Burchnall and Chaundy, Krichever (\cite{Kr}, \cite{Kr1}) gave a geometric classification of rank $r$ algebras of "generic position" in terms of spectral data. Drinfeld \cite{Dr} gave an algebro-geometric reformulation of Krichver's results which was improved later by Mumford \cite{Mum1}. Later Verdier and Mulase gave a classification of all rank $r$ algebras. Mulase's classification was a natural improvement of the theorems of Krichever and Mumford, Verdier used other ideas and proposed a classification in terms of parabolic structures and connections of vector bundles defined on curves. It is important to notice that the constructions of Krichever, Mumford and Mulase are essentially constructible in both directions, i.e. for a given ring of commuting operators one can construct a geometric data, and vice versa. This leads to a possibility to use this method for constructing examples of commuting operators. 

After their work, many attempts have been made to classify algebras of commuting partial differential operators in several variables. There are several approaches to this problem (see e.g. review \cite{Pr} and references therein). One of the methods is based on the approach of Nakayashiki (see \cite{Na}, \cite{Mi}, \cite{Roth} and references therein) and the other method uses ideas from differential algebra (see \cite{Pr} and references therein). Nevertheless, the methods above don't lead to a classification, and Nakayashiki's approach leads to rings of commuting  partial differential operators with matrix (not of dimension 1) coefficients. 

The classification of ordinary differential operators can be considered as a part of the KP theory that relates several mathematical objects: solutions of the KP equation (or of the KP hierarchy), geometrical (spectral) data, rings of ordinary differential operators, points of the Sato grassmanian. In works  \cite{Pa}, \cite{Pa1}, \cite{Os}, \cite{Ku},\cite{Ku1}, \cite{Zhe} several pieces of analogous KP theory in dimension two are developed: there are analogues of the KP hierarchy, geometrical data, Jacobians.  

The solution of the classification problem of commutative rings of operators we are proposing in this paper uses our original approach based on some ideas of Parshin (see \cite{Pa}, \cite{Pa1}) and works cited above, and is a natural generalization of the theorems of Krichever, Mumford and Mulase, and is constructible in both ways. On the other hand, it generalizes the approach of M.Sato in dimension one. The methods used in this paper could be generalized also to higher dimension, and we plan to describe the general case in another paper. The reason to describe first carefully dimension two case is that this case is applicable to already investigated theory of ribbons (see \cite{Ku},\cite{Ku1}) and theory of generalized Parshin-KP's hierarchies (see \cite{Pa1}, \cite{Zhe}), which have been developed only for dimension 2 case.

As a result we obtain a classification of commutative subrings (satisfying certain mild conditions, see theorems \ref{schurpair} and \ref{dannye2}) in the ring of completed differential operators $\hat{D}$ (see subsection \ref{defin1}) that contain the ring of partial differential operators 
$k[[x_1,x_2]][\partial_{x_1},\partial_{x_2}]$, where $k$ is a field of characteristic zero, as a dense subring. The operators from the ring $\hat{D}$ contain all usual partial differential operators, and difference operators as well. They are also linear and act on the ring of germs of analytical functions. 

Such commutative subrings include as a particular case all commutative subrings of partial differential operators (satisfying the same mild conditions, see theorem \ref{dannye2}) because of the following result on "purity" (see proposition \ref{purity}): any commutative subring in $\hat{D}$ containing such a ring of partial differential operators is itself a ring of partial differential operators. Thus, we obtain in a sense also a classification of commutative subrings of partial differential operators, although  there is a problem of finding extra conditions on the classifying data describing rings of partial differential operators between rings of operators in $\hat{D}$, see remark \ref{zakl3}. 

We would like to emphasize that the ring $\hat{D}$ naturally appears in our approach of generalization of the KP theory to higher dimension (cf. remark \ref{posl}). In dimension one there is no need to introduce it. As in one-dimensional case, one can introduce a notion of formal Baker-Akhieser function (cf. \cite[Introduction]{ZhM}), which in case of rings of partial differential operators satisfying certain conditions is an analogue of the Baker-Akhieser function considered in \cite{Kr} (see remark \ref{BA}). The explicit formula for this Baker-Akhieser  function uses local parameters at the point $P$ of the geometrical data (see definition \ref{geomdata}). We emphasize that this data did not appear in earlier approaches.

\bigskip

The classification we are giving here is divided in three steps. First we reduce the problem to the case of rings satisfying certain special properties ($1$-quasi elliptic rings, see definition \ref{elliptic}). Then we classify a bigger class of $\alpha$-quasi elliptic rings: namely, all such rings in a completed ring of differential operators (see subsection \ref{defin1}, definition \ref{elliptic}). We classify them in terms of pairs of subspaces (generalized Schur pairs, see definitions \ref{sch}, \ref{schurdata}). This classification uses a generalization of M.Sato's theory (see \cite{Sa}, \cite{SN}), and is constructible in both ways. After that we classify generalized  Schur pairs in terms of generalized geometric data (see definition \ref{geomdata}). On the one hand side, the data is a natural generalization of the geometric data in one dimensional case, on the other hand, it is a slight modification of the geometric data of Parshin  \cite{Pa} and Osipov \cite{Os}. The exposition of the last two steps of our classification follows closely to the exposition of the corresponding results in the work of Mulase \cite{Mu}. In particular, as the last step of the classification we introduce two categories, the category of Schur pairs (definition \ref{schurcategory}) and the category of geometric data (definition \ref{geomcategory}), and show their anti-equivalence. These categories are natural generalizations of the corresponding categories from \cite{Mu}. 

\bigskip

The paper is organized as follows. In section 2 we recall some known facts about rings of partial differential operators, introduce new notation and develop a generalization of the M.Sato theory. In section 3 we realize three steps of the classification described above. In section 4 we announce some examples (omitting all calculations that will appear in \cite{Ku3}) and explain how known examples of commuting partial differential operators (such as operators corresponding to quantum Calogero-Moser system or rings of quasi invariants, see \cite{Ch}, \cite{FV}, \cite{EG}, \cite{BEG}, \cite{FV2}) fit into the proposed classification. At the end of this section we prove a theorem about algebraic-geometric properties of maximal commutative subrings of partial differential operators in two variables; in particular, we show that all such rings must be Cohen-Macaulay. 

Some applications of constructions described in this paper to the theory of ribbons (see \cite{Ku},\cite{Ku1}) and theory of generalized Parshin-KP's hierarchies (see \cite{Pa1}, \cite{Zhe}), as well as several explicit examples of commuting  operators, will appear in a separate paper (see \cite{Ku3}), part of which is a recent work \cite{Ku4} (cf. also work \cite{ZhM} for a comparison with Baker-Akhieser-modules-approach). 

\bigskip

{\bf Acknowledgements.} I am grateful to Herbert Kurke for his important and useful comments and exposition improvements made on the earlier version of this paper. I am also grateful to Denis Osipov for many stimulating discussions and useful suggestions. I would like to thank the MFO at Oberwolfach for the excellent working
conditions, where several improvements of this work has been done.

\section{Analogues of the Sato theory in dimension 2} 

\subsection{General setting}

\subsubsection{Generalities}

Let $R$ be a commutative $k$-algebra, where $k$ is a field of characteristic zero. 

Then we have the filtered ring $D(R)$ of $k$-linear differential operators and the $R$-module $\Der (R)$ of derivations:
$$
D_0(R)\subset D_1(R)\subset D_2(R)\subset \ldots ,\mbox{\quad } D_i(R)D_j(R)\subset D_{i+j}(R), \mbox{\quad } \Der (R)\subset D_1(R)
$$

$D_i(R)$ are defined inductively as sub-$R$-bimodules of $\End_{k}(R)$; by definition $D_0(R)=\End_{R}(R)=R$, 
$$
D_{i+1}(R)=\{P\in {\End}_k(R)| \mbox{\quad such that for all $f\in R$ $[P,f]\in \Der (R)$}\}.
$$

Then we can form the graded ring 
$$
gr (D(R))=\oplus_{i=0}^{\infty}D_i(R)/D_{i-1}(R)\mbox{\quad} (D_{-1}(R)=0)
$$
and for $P\in D_i(R)$ the {\it principal symbol} $\sigma_i(P)=P \mod D_{i-1}(R)$. For $P\in D_i$, $Q\in D_j$ we have $\sigma_i(P)\sigma_j(Q)=\sigma_{i+j}(PQ)$, 
$[P,Q]\in D_{i+j-1}$, hence $gr (D(R))$ is a commutative graded $R$-algebra with a Poisson bracket 
$$\{\sigma_i(P), \sigma_j(Q)\}=\sigma_{i+j-1}([P,Q])$$ 
with the usual properties.  

\subsubsection{Coordinates}

\begin{defin}
We say that $R$ has a system of coordinates $(x_1,\ldots ,x_n)\in R^n$ if 
\begin{enumerate}
\item
The map 
$$
{\Der}_k (R) \rightarrow R^n, \mbox{\quad} D\mapsto (D(x_1), \ldots ,D(x_n))
$$
is bijective.
\item
$\cap_{D\in {\Der}_k(R)}\Ker (D)=k$.
\end{enumerate}
\end{defin}
In this case there are $\partial_1,\ldots ,\partial_n\in {\Der}_k(R)$ satisfying 
$$
\partial_i(x_j)=\delta_{ij}, \mbox{\quad} \Ker (\partial_1)\cap\ldots\cap\Ker (\partial_n)=k.
$$
Then $\Der (R)$ is a free $R$-module with generators $\partial_1,\ldots ,\partial_n$ and we have $[\partial_i,\partial_j]=0$. One checks (by induction on the grade) that 
$$
gr (D(R))\simeq R[\xi_1, \ldots ,\xi_n]\mbox{\quad by } \xi_i\mapsto \partial_i \mod D_0(R)\in gr_1(D(R))
$$
and that for $P\in D_i(R)$, $Q\in D_j(R)$ we have
$$
\{\sigma_i (P), \sigma_j(Q)\} = \sum_{v=1}^n \frac{\partial \sigma_i(P)}{\partial \xi_v}\partial_v(\sigma_j(Q))- \sum_{v=1}^n \frac{\partial \sigma_j(Q)}{\partial \xi_v} \partial_v(\sigma_i(P))
$$
(where we have extended $\partial_v$ to $R[\xi_1,\ldots ,\xi_n]$ by $\partial_v(\xi_l)=0$). 

The system $(x_1,\ldots ,x_n,\xi_1,\ldots ,\xi_n)$ is called a {\it canonical coordinate system. } A typical example of a ring with a coordinate system is the ring $k[x_1,\ldots ,x_n]$ or $k[[x_1,\ldots ,x_n]]$, where in the last case we have to restrict ourself to the ring of continuous differential operators and to the space of continuous derivations with respect to the usual topology on
$k[[x_1,\ldots ,x_n]]$ given by the maximal ideal. The ring $k[[x_1,\ldots ,x_n]]$ will be important for the main part of the article.

\subsubsection{Coordinate change}

If $(y_1,\ldots , y_n)$ is another coordinate system, we get a new basis $(\partial_1',\ldots ,\partial_n')$ of ${\Der}_k(R)$ and the change of coordinates is related by the matrix 
$$
\left(
\begin{array}{ccc}
\partial_1(y_1)&\ldots &\partial_n(y_1)\\
\partial_1(y_2)&\ldots &\partial_n(y_2)\\
\vdots &\ddots &\vdots \\
\partial_1(y_n)& \ldots &\partial_n(y_n)
\end{array}
\right)
=M
$$
as $(\partial_1', \ldots ,\partial_n')M=(\partial_1, \ldots ,\partial_n)$, 
$(\xi_1', \ldots ,\xi_n')M=(\xi_1, \ldots ,\xi_n)$. 

\begin{defin}
\label{defin7}
If we have fixed a coordinate system $(x_1,\ldots x_n)$ we get besides the usual order function 
$$
\Ord (P)=\inf \{n| P\in D_n(R)\}
$$
and the usual filtration a finer $\Gamma$-filtration with $\Gamma =\dz^n$ endowed with the anti lexicographical order as an ordered group.  

Every $P\in D(R)$ can be expressed as 
$$
P=\sum_{finite} p_{i_1\ldots i_n} \partial_1^{i_1}\ldots \partial_n^{i_n}
$$
and $p_{i_1\ldots i_n} \partial_1^{i_1}\ldots \partial_n^{i_n}$ with $p_{i_1\ldots i_n}\neq 0$ are called {\it terms of $P$}. 

The {\it highest term} is the term $p_{m_1\ldots m_n} \partial_1^{m_1}\ldots \partial_n^{m_n}$ with $(m_1,\ldots   ,m_n)>(i_1,\ldots ,i_n)$ for every other term. 
\end{defin}
\begin{defin}
\label{gorder}
The element $(m_1,\ldots ,m_n)\in \Gamma$ is called $\Gamma$-order $\ord_{\Gamma}(P)$ and the term $p_{m_1\ldots m_n}\partial_1^{m_1}\ldots \partial_n^{m_n}$ is called the {\it highest term} $\HT (P)$. 
\end{defin}

Clearly, we have $\ord_{\Gamma}(PQ)=\ord_{\Gamma}(P)+\ord_{\Gamma}(Q)$ and
$\ord_{\Gamma}(P+Q)\le \max\{ \ord_{\Gamma}(P), \ord_{\Gamma}(Q)\}$ with equality if $\ord_{\Gamma}(P)\neq \ord_{\Gamma}(Q)$. Also $\HT (PQ)=\HT (P)\HT (Q)$ and $\HT (P+Q)=\HT (P)$ if $\ord_{\Gamma} (P)>\ord_{\Gamma} (Q)$. 

\subsubsection{Extensions of the ring $D(R)$}

There are several ways to extend the ring $D=D(R)$ to a ring $E\supset D$ either with an extension of the filtration $(D_n)_{n\ge 0}$ to a filtration $(E_n)_{n\in\sdz }$ with $gr (E)$ commutative such that $P\in E$ is invertible in $E$ iff $\sigma_{\Ord (P)}(P)$ is invertible in $gr (E)$ (formal micro differential operators) or to another filtered ring with an extension of the $\Gamma$-filtration and the highest term map (given by the choice of a coordinate system) with the property: $P$ is invertible in $E$ if and only if the coefficient of $\HT (P)$ is invertible in $R$ ({\it formal pseudo-differential operators}). 

We describe here formal pseudo-differential operators: 
$E=R((\partial_1^{-1}))\ldots ((\partial_n^{-1}))$ (cf. \cite{Pa1}). 

This ring can be defined iteratively, starting by defining the ring $A((\partial^{-1}))$, where $A$ is an associative not necessary commutative ring with a derivation $d$. The ring  $A((\partial^{-1}))$ is defined as a left $A$-module of all formal expressions 
$$
L=\sum_{i>-\infty}^na_i\partial^i, \mbox{\quad} a_i\in A.
$$
A multiplication can be defined according to the Leibnitz rule:
$$
(\sum_ia_i\partial^i)(\sum_jb_j\partial^j)=\sum_{i,j,k\ge 0}C_i^ka_id^k(b_j)\partial^{i+j-k}.
$$
Here we put 
$$
C_i^k=\frac{i(i-1)\ldots (i-k+1)}{k(k-1)\ldots 1} \mbox{ if $k>0$, $C_i^0=1$}.
$$
It can be checked that $A((\partial^{-1}))$ will be again an associative ring. 

For an element $P\in E$ we formally write $P= \sum_{\i\in \Gamma}r_{\i}\partial_1^{i_1}\ldots\partial_n^{i_n}$ (here some of the  coefficients $r_{\i}$ can be equal zero). 

Because of definition, there is a highest term $\HT (P)=r_{m_1\ldots m_n} \partial_1^{m_1}\ldots\partial_n^{m_n}$ with $r_{m_1\ldots m_n}\neq 0$, where $(m_1,\ldots ,m_n)\ge (i_1,\ldots ,i_n)$ if $r_{i_1,\ldots ,i_n}\neq 0$. It has the same properties as the highest term on $D(R)$. We define $\ord_{\Gamma}(P)=(m_1,\ldots ,m_n)$.   
\begin{nt}
If $P\in E$ and if $\HT (P)=r_{m_1\ldots m_n}\partial_1^{m_1}\ldots \partial_n^{m_n}$ then $r_{m_1\ldots m_n}$ is invertible in $R$ if and only if $P$ is invertible in $E$. 
\end{nt}

\begin{defin}
\label{action}
Let $R$ be a ring with a system of coordinates $(x_1,\ldots ,x_n)$, let $M=(x_1R+\ldots +x_nR)$ be an ideal and $R/M=k$. We get a right ideal $x_1E+\ldots +x_nE\subset E$ and a right $E$-module $E/(x_1E+\ldots +x_nE)\simeq k((z_1))\ldots ((z_n))$ (isomorphic as $k$-vector spaces) which gives a right $E$-module structure on $V=k((z_1))\ldots ((z_n))$. We also get an isomorphism $gr (R)\simeq k[x_1,\ldots ,x_n]$ (here the filtration in $R$ is taken to be generated by powers of $M$), and we'll denote by $\bar{a}$ the image of element $a\in R$ in $gr (R)$. 

Denote by $M_i$ the ideal $x_iR$ and for $a\in R$ define 
$$
\ord_{M_i}(a)=\sup\{n| a\in M_i^n\},\mbox{\quad} \ord_M(a)=\sup\{n| a\in M^n\}
$$ 
In analogy with definitions \ref{defin7}, \ref{gorder} on the ring $gr (R)$ is defined a finer $\Gamma$-filtration with $\Gamma =\dz^n$ endowed with the anti lexicographical order and the $\Gamma$-order function $\ord_{\Gamma}$: if $\bar{r}=\sum \bar{r}_{i_1\ldots i_n}x_1^{i_1}\ldots x_n^{i_n}\in gr (R)$, 
$$
\ord_{\Gamma}(\bar{r})=\min\{(i_1,\ldots,i_n)\in \Gamma | \bar{r}_{i_1\ldots i_n}\neq 0\}.
$$

Now for $r\in R$ define 
$$
\ord_{M_1,\ldots ,M_n}(r)=\ord_{\Gamma}(\bar{r}), 
$$
and for $P\in E$ define 
$$
\ord_{M_1,\ldots ,M_n}(P)=\min_{\i\in \Gamma}\{(\ord_{M_1,\ldots ,M_n}(r_{\i})\in \Gamma\}. 
$$
Below we will write $z^{\i}$ ($\partial^{\i}$) instead of $z_1^{i_1}\ldots z_n^{i_n}$ ($\partial_1^{i_1}\ldots \partial_n^{i_n}$) for a multi index $\i =(i_1,\ldots ,i_n)$. For $P\in E$ denote by $P(0)$ the image of $P$ modulo $M$ in $V$. 
\end{defin}

Note that $\ord_M, \ord_{M_i}, \ord_{M_1,\ldots ,M_n}$ are (pseudo)-valuations. 

\begin{prop}
\label{prop1} 
If $W_0=k[z_1^{-1},\ldots ,z_n^{-1}]\subset V$ then $D\subset E$ is characterized as $D=\{A\in E | W_0A\subseteq W_0\}$.
\end{prop}

\begin{proof} Clearly, $D\subset \{A\in E | W_0A\subset W_0\}$. For $A\in E$  denote by $A_+$ the sum of all monomials in $A$ belonging to $D$, and set $A_-=A-A_+$. If $A\in E$ and $A\notin D$ then  $A_-\neq 0$. In this case   we have 
$$
0\neq z^{-\ord_{M_1,\ldots ,M_n}(A_-)}A_-= \partial^{\ord_{M_1,\ldots ,M_n}(A_-)}(A_-)(0)\notin W_0,
$$
where the equality holds since $\partial^{\i}(A_-)(0)=0$ for $\i <\ord_{M_1,\ldots ,M_n}(A_-)$. 
Since $z^{-\ord_{M_1,\ldots ,M_n}(A_-)}A_+ \in W_0$, we obtain $z^{-\ord_{M_1,\ldots ,M_n}(A_-)}A\notin W_0$. So, if $A$ preserves $W_0$, $A$ must be in $D$.  
\end{proof}

\subsubsection{Completion}
\label{defin1}

Consider a ring $R$ endowed with a $M$-adic topology ($M$ ideal in $R$) which is complete: $R=\limproj_{n\ge 0}(R/M^n)$. 

If $N\subset D$ is a subalgebra we define for each sequence in $MD$, $(P_n)_{n\in \sdn}$, such that $P_n(R)$ converges uniformly in $R$ (i.e. for any $k>0$ there is $N>0$ such that $P_n(R)\subseteq M^k$ for $n\ge N$) a $k$-linear operator $P: R\rightarrow R$ by
$$
P(f)=\limind_{n\rightarrow \infty}\sum_{v=0}^nP_v(f), \mbox{\quad} P:=\sum_n P_n
$$
(this might be no longer a differential operator). 

Denote by $\hat{N}$ the algebra of these operators. One can easily check that it is associative. 

We also define 
$$
\hat{D}_N=\mbox{ algebra generated by $\hat{N}$ and $D$}.
$$  
If $(x_1,\ldots , x_n)$ is a coordinate system and $M=x_1R+\ldots +x_nR$ we can consider the algebra $\hat{D}_m:=\hat{D}_N$ given by $N=R[\partial_1,\ldots , \partial_m]$. 

The operator $P$ in $\hat{D}_m$ is uniquely defined by the sequence $p_{i_1\ldots i_m}=P(x_1^{i_1}\ldots x_m^{i_m}/i_1!\ldots i_m!)$. The elements of $\hat{D}_m$ correspond precisely to those sequences $(p_{\i}=p_{i_1\ldots i_m})_{\i\in\sdn^m}$ which converge to zero in the $M$-adic topology for $|\i |=i_1+\ldots +i_m \rightarrow \infty$. Namely,
$$
(p_{\i})\longleftrightarrow P=\sum_{\i}p_{\i}\partial_1^{i_1}\ldots \partial_m^{i_m}= \lim_{n\rightarrow \infty}(\sum_{|\i |\le n}p_{\i}\partial_1^{i_1}\ldots \partial_m^{i_m}).
$$

Then we define 
$$
\hat{D}_{m,n-m}=\mbox{ algebra generated by $\hat{D}_m$ and $D$}=\hat{D}_m[\partial_{m+1},\ldots ,\partial_n]
$$
and in the usual way 
$$
\hat{E}_{m,n-m}=\hat{D}_m((\partial_{m+1}^{-1}))\ldots ((\partial_n^{-1}))\supset R[\partial_1, \ldots , \partial_m]((\partial_{m+1}^{-1}))\ldots ((\partial_n^{-1}))=E_{m,n-m}
$$
\begin{ex}
Let's give another description of the rings $\hat{D}_m, \hat{D}_{m,n-m}$ in the case we will be interested in this paper. Namely, let $R=k[[x_1,x_2]]$. Then the coordinate system in $R$ is $(x_1,x_2)$ and $M=(x_1,x_2)$ is a maximal ideal. Then define the set
\begin{multline}
\hat{D}_1=\{a=\sum_{q\ge 0} a_{q}\partial_1^q\mbox{\quad }|  a_q\in k[[x_1,x_2]] \mbox{ and for any $N\in \dn$ there exists $n\in \dn$ such that }\\
\mbox{$\ord_{M}(a_m)>N$ for any $m\ge n$}\}. 
\end{multline}

Define
$$
\hat{D}_{1,1}=\hat{D}_1[\partial_2], \mbox{\quad} \hat{E}_{1,1}=\hat{D}_1((\partial_2^{-1})).
$$

\begin{lemma}
\label{lemma4}
The sets $\hat{D}_1\subset \hat{D}_{1,1}\subset \hat{E}_{1,1}$ are associative rings with unity. 
\end{lemma}

\begin{proof} Obviously, the set $\hat{D}_1$ is an abelian group. The multiplication of two elements is defined by the following formula: for two series $A=\sum_{q\ge 0} a_{q}\partial_1^q$, $B=\sum_{q\ge 0} b_{q}\partial_1^q$  
$$
AB=\sum_{q\ge 0}g_q\partial_1^q, \mbox{\quad where } g_q=\sum_{k\ge 0}\sum_{l\ge 0}C_{k}^la_k\partial_1^l(b_{q+l-k}),
$$
where we assume $b_i=0$ for $i<0$. Each coefficient $g_q$ is well defined, because  for each $N$ there are only finite number of $a_k$ with $\ord_M(a_k)<N$ and for each $k$ there are only finite number of $C_{k}^l\neq 0$. 

For any $N$ there is $n$ such that $\ord_M(a_m)>N$ for any $m\ge n$, and there is $n_1$ such that $\ord_M(b_m)>N+n$ for any $m\ge n_1$. Then for any $q\ge n_1+n$ and any $k<n$, $0\le l\le k$ we have $\ord_M(\partial_1^l(b_{q+l-k}))\ge \ord_M(b_{q+l-k})-l>N$. Therefore, $\ord_M(g_q)> N$ for any $q\ge n_1+n$. So, the multiplication is well defined in $\hat{D}_1$. 
The distributivity is obvious, and the associativity can be proved by the same arguments as in \cite[ch.III, \S 11]{Ma}.

The proof for $\hat{D}_{1,1}, \hat{E}_{1,1}$ is the same.  
\end{proof}
\end{ex}

The action of $E_{m,n-m}$ on $V=k ((z_1))\ldots ((z_n))$ does not extend to an action of $\hat{E}_{m,n-m}$ on $V$, but partially it extends. To explain this we introduce the notion:
\begin{defin}
\label{defin2,5}
Terms of $v=\sum_{(i_1,\ldots ,i_n)}v_{i_1\ldots i_n}z_1^{i_1}\ldots z_n^{i_n}$ are the elements $v_{i_1\ldots i_n}z_1^{i_1}\ldots z_n^{i_n}$ with $v_{i_1\ldots i_n}\neq 0$, we order them by the anti lexicographical order on $\Gamma$, $\ord_{\Gamma}(z_1^{i_1}\ldots z_n^{i_n})=(i_1,\ldots ,i_n)$. Each $v$ has a {\it lowest term } $\LT (v)$ (term of lowest order) whose order is called the $\Gamma$-order of $v$, $\ord_{\Gamma}(v)$. 
\end{defin}

Note that $\ord_{\Gamma}$ on $V$ is a discrete valuation of rank $n$. For an action of $E$ on $V$ we have 
$$
\ord_{\Gamma}(vP)\ge \ord_{\Gamma}(v)-\ord_{\Gamma}(P)
$$ 
with equality if and only if $\HT (P)$ has an invertible coefficient in $R$. 

Recall one definition from the theory of multidimensional local fields: 
\begin{defin}
\label{topology}
Starting with the discrete topology on the field $k$ we define a topology on the space $V$ iteratively as follows. 

If $F=k((z_1))\ldots ((z_{k-1}))$ has a topology, consider the following topology on $K=F((z_k))$. For a sequence of neighbourhoods of zero $(U_i)_{i\in \sdz}$ in $F$, $U_i=F$ for $i\gg 0$, denote $U_{\{U_i\}}=\{\sum a_iz_k^i: a_i\in U_i\}$. Then all $U_{\{U_i\}}$ constitute a base od open neighbourhoods of zero in $F((z_k))$. In particular, a sequence $u^{(n)}=\sum a_i^{(n)}z_k^i$ tends to zero if and only if there is an integer $m$ such that $u^{(n)}\in z_k^mF[[z_k]]$ for all $n$ and the sequences $a_i^{(n)}$ tend to zero for every $i$.  
\end{defin}

Now consider the following closed subspaces in $V$: 
$$
W_{m,n-m}=k[z_1^{-1},\ldots ,z_m^{-1}]((z_{m+1}))\ldots ((z_n)).
$$
One can easily check that the action of $E_{m,n-m}$ on $W_{m,n-m}$ extends to the action of $\hat{E}_{m,n-m}$ in the same way via the isomorphism 
$\hat{E}_{m,n-m}/M\hat{E}_{m,n-m}\simeq k[z_1^{-1}, \ldots ,z_m^{-1}]((z_{m+1}))\ldots ((z_n))$. At the same time, the action of $\hat{E}_{m,n-m}$ on say $\partial_1^{-1}$ (if $m\ge 1$) is not correctly defined. 

\begin{nt}
\label{nt1}
 Note that the elements of the ring $\hat{D}_{m,n-m}$ can be viewed as "extended" differential operators, because they act on the elements of the ring $R$ in the same way as the  usual differential operators. 
 
 We note also that the ring $\hat{D}_{m,n-m}$ has zero divisors (see examples in \cite{Ku3}). 
\end{nt}

\begin{prop}
\label{prop2} 
We have $\hat{D}_{m,n-m}=\{A\in \hat{E}_{m,n-m} | W_0A\subset W_0\}$ (here $W_0=k[z_1^{-1},\ldots ,z_n^{-1}]\subset W_{m,n-m}$).
\end{prop}

The proof is the same as the proof of proposition \ref{prop1}.\\

\subsubsection{Further remarks}

In this section we would like to make several comments on our definitions of rings and subspaces introduced above. 

In case of dimension one, i.e. for the rings of ordinary differential operators $D$ and pseudo-differential operators $E$, the classical KP-theory deals with a decomposition $E=E_+\oplus E_-$, where $E_+=D$. This decomposition is used then to define a KP system and develop the KP theory.
 
In \cite{Pa1} Parshin introduced an analogue of the classical KP system in higher dimensions using an analogue of the decomposition above. This system and its modifications studied later in \cite{Zhe}. 

Let's illustrate how our rings are related with a decomposition of the ring $E$ in two dimensional case. Consider the ring $E=k[[x_1,x_2]]((\partial_1^{-1}))((\partial_2^{-1}))$.

\begin{defin}
\label{W_l}
We define a vector space $W_l$ as a closed vector subspace in the field $k((z_1))((z_2))$ generated by monomials $z_1^nz_2^m$, $n\le 0$, $n,m\in \dz$. 
\end{defin}

Now we want to define the decomposition:
$$
E=E_+^l\oplus E_-^l.
$$
\begin{defin}
\label{+-}
We define the "$+$" part $E_+$ ({\it $l$-differential operators}) as follows:
$$
E_+^l=\{A\in E | W_lA\subset W_l\},
$$
the "$-$" part: 
$$
E_-^l=k[[x_1,x_2]]\partial_1^{-1}[[\partial_1^{-1}]]((\partial_2^{-1}))
$$
\end{defin}
\begin{lemma}
\label{lemma1}
The set $E_+^l$ is an associative  ring with unity; $E_+^l=k[[x_1,x_2]][\partial_1]((\partial_2^{-1}))$.
\end{lemma}

\Proof . The first claim follows from the second. 

The set $E_+^l$ is, obviously, an Abelian group. It is a monoid under the multiplication in the ring $E$, because for any elements $A,B\in E_+^l$ and for any $w\in W_l$ $w(AB)=(wA)B\in W_l$.

The associativity and distributivity of the multiplication follow from the corresponding properties in the ring $E$. Clearly, $k[[x_1,x_2]][\partial_1]((\partial_2^{-1}))\in  E_{+}^l$. 

The rest of the proof follows from the following two lemmas.

\begin{lemma}
\label{lemma2}
The set $E_-^l$ is an associative  ring. A non-zero operator from this set does not belong to $E_+^l$. 
\end{lemma}

{\Proof } The proof of the first statement is clear. The proof of the second statement is analogues to the proof of proposition \ref{prop1}. 

\begin{lemma}
\label{lemma3}
There exists a unique decomposition 
$$
E=E_+^l\oplus E_-^l 
$$
\end{lemma}

The proof is clear. \\

In particular, we obtain that $E_+^l=E_{1,1}$. {\it {\bf Further we will often write $E_+$ instead of $E_+^l$ and $E_{1,1}$, and $\hat{E}_+$ instead of $\hat{E}_{1,1}$. Also we will write $\hat{D}$ instead of $\hat{D}_{1,1}$.}}

\subsection{An analogue of the Sato theorem in dimension 2}

We consider in this section the ring $E=k[[x_1,x_2]]((\partial_1^{-1}))((\partial_2^{-1}))$. 

Recall the definition of the support of a $k$-subspace in the space $k((z_1))((z_2))$. 
\begin{defin}{(\cite{ZO})}
\label{defin2}
The support of a $k$-subspace $W$ from the space $k((z_1))((z_2))$ is the closed $k$-subspace $\Sup (W)$ in the space $k((z_1))((z_2))$ generated by $\LT (a)$ for all $a\in W$. 
\end{defin}

In dimension 1 there is the Sato theorem (see for example \cite{Mu}, appendix) that describes the correspondence between points of the big cell of the Sato grassmanian and operators from the Volterra group. We can prove the following analogue of this theorem in dimension two.

\begin{theo}
\label{theo1}
For any closed $k$-subspace $W\subset  k[z_1^{-1}]((z_2))$ with $\Sup(W)=W_0=k[z_1^{-1},z_2^{-1}]$  
there exists a unique operator $S=1+S^-$, where $S^-\in \hat{D}_1[[\partial_2^{-1}]]\partial_2^{-1}$, such that 
$W_0S=W $. 
\end{theo}

\begin{proof} Note that any operator $S=1+S^-$, where $S^-\in \hat{D}_1[[\partial_2^{-1}]]\partial_2^{-1}$, is invertible, $S^{-1}  =1-S^-+(S^-)^2-\ldots$. If we have two operators $S_1,S_2$ of such type, then $S_1S_2 -1\in \hat{D}_1[[\partial_2^{-1}]]\partial_2^{-1}$. 

Uniqueness: if there are two such operators, $S,S'$, then $W_0=W_0S'S^{-1}$, hence by proposition \ref{prop2} $S'S^{-1}\in \hat{D}$. So, $S'S^{-1}=1$.

Existence: For any $(k,l)\in \dz_+\oplus \dz_+$ we must have $z_1^{-k}z_2^{-l}S\in W$. From definition of the action we have  
\begin{equation}
\label{eq1}
z_1^{-k}z_2^{-l}S= \partial_1^k\partial_2^l(S)(0)+\sum ,
\end{equation}
where $\sum$ is the finite sum of elements of the following type: $const\cdot z_1^{-m}z_2^{-n}\partial_1^p\partial_2^q(S)(0)$ with 
$m\le k$, $n\le l$, $p\le k$, $q\le l$ and $m+p=k$, $n+q=l$. 

Let's call the series $\partial_1^k\partial_2^l(S)(0)$ by the $(k,l)$-slice of $S$. Note that $S$ is uniquely defined by its $(k,l)$-slices for all $k,l\ge 0$: namely, the $(k,l)$-slice is the series of coefficients at $x_1^kx_2^l$,
$$
S=\sum_{k=0}^{\infty }\sum_{l=0}^{\infty }x_1^kx_2^l\partial_1^k\partial_2^l(S)(0).
$$
From (\ref{eq1}) follows that the $(k,l)$-slice of $S$ is uniquely defined by the element $z_1^{-k}z_2^{-l}S\in W$ and by the $(p,q)$-slices with $(p,q)<(k,l)$. 

We know that $\ord_{\Gamma} (z_1^{-k}z_2^{-l}S)=(k,l)$. We can take a basis $\{w_{i,j}, i,j\ge 0\}$ in $W$ with the property $w_{i,j}=z_1^{-i}z_2^{-j}+w_{i,j}^-$, where $w_{i,j}^-\in k[z_1^{-1}][[z_2]]z_2$ (note that such a basis is uniquely defined). Then, on the one hand side, we have  
$$
z_1^{-k}z_2^{-l}S=\sum_{0\le (i,j)\le (k,l)}b_{i,j}w_{i,j}, \mbox{\quad} b_{i,j}\in k.
$$
On the other hand side, we have 
$$
\sum = \sum_{0\le (i,j)\le (k,l)}a_{i,j}z_1^{-i}z_2^{-j} +\sum_-, \mbox{\quad where } \sum_-\in k[z_1^{-1}][[z_2]]z_2,
$$
and $\partial_1^k\partial_2^l(S)(0)\in k[z_1^{-1}][[z_2]]z_2$. 
So, we must have $b_{i,j}= a_{i,j}$, and therefore the element $z_1^{-k}z_2^{-l}S$ is uniquely defined by $\sum$. 

So, starting with $(k,l)=(0,0)$, we find first the $(0,0)$-slice, then, by induction, we find  the $(k,0)$-slice  for each $k>0$, and then, again by induction, we find the $(k,l)$-slice for each $(k,l)$.  
\end{proof}
 
\subsection{Several facts about partial differential operators}
\label{facts}

Further we will need several technical statements about rings of differential operators. For convenience we'll recall several known facts in the next subsection.

\subsubsection{Characteristic scheme}

If $J\subset D$ is a left ideal we get a homogeneous ideal $\langle \sigma_i(P), P\in J\rangle$ in $gr (D)$ and a subscheme defined by this ideal in either $\Spec (gr(D))$ or $\Proj (gr(D))$. Both are called the characteristic subscheme $\Ch (J)$. We consider the characteristic subscheme in $\Proj (gr(D))$. 

If we have a coordinate system, we get $\Proj (gr(D))=\Proj (R[\xi_1,\ldots ,\xi_n])=\Spec (R)\times_k\dpp_k^{n-1}$. 
Consider the case of the ideal $J=PD$, where $P$ is an operator with $\Ord (P)=m$. If $\sigma_m(P)\in k[\xi_1,\ldots ,\xi_n]$ we say that {\it the principal symbol is constant.} In this case the characteristic scheme is essentially given by the divisor of zeros of $\sigma_m(P)$ in $\dpp^{n-1}$, we call it $\Ch_0(P)$. It is unchanged by a $k$-linear change of coordinates. 

\begin{lemma}
\label{wellness}
If $P_1,\ldots P_n$ are operators with constant principal symbols (with respect to a coordinate system $(x_1,\ldots ,x_n)$) and if $\det (\partial \sigma (P_i)/\partial \xi_j)\neq 0$ then any operator $Q$ with $[P_i,Q]=0$, $i=1,\ldots , n$ has also a constant principal symbol.
\end{lemma}

\begin{proof} We have 
$$
0=\{\sigma (P_i),\sigma (Q)\} =\sum_j \frac{\partial (\sigma (P_i))}{\partial \xi_j} \partial_j (\sigma (Q))
$$
for $i=1,\ldots ,n$. Since $\det (\partial \sigma (P_i)/\partial \xi_j)\in k[\xi_1,\ldots ,\xi_n]$ is not zero, we infere $\partial_j(\sigma (Q))=0$  for $j=1,\ldots n$, hence $Q$ has constant principal symbol with respect to $(x_1,\ldots ,x_n)$.  
\end{proof}

\begin{prop}
\label{techn5.2}
If $P_1,\ldots ,P_n\in D$ are commuting operators of positive order with constant principal symbols with respect to coordinates $(x_1,\ldots ,x_n)$, and if the characteristic divisors of $P_1,\ldots ,P_n$ have no common point (in $\dpp^{n-1}$), then there hold
\begin{enumerate}
\item 
If $B$ is a commutative subring in $D$ containing $P_1,\ldots ,P_n$ then $\gr (B)\subset k[\xi_1,\ldots ,\xi_n]$. 
\item
Any such subring is finitely generated of Krull dimension $n$, and also $\gr B$ is finitely generated of Krull dimension $n$. 
\end{enumerate}
\end{prop}

\begin{nt}
The items $1$ and partially item $2$ follow from \cite[Ch.III, \S 2.9, Prop.~10]{Bu}. The item $2$ was proved in~\cite{Kr} by Krichever in connection with integrable systems. We give here an alternative proof in the spirit of pure commutative algebra.

In section \ref{reduction} we will show that in fact there is a unique maximal commutative subring in $D$ under assumptions of lemma.
\end{nt}

\begin{proof} If $m_i=\deg (P_i)$ and $Q\in B\cap D_m$ then 
$$
0=\{\sigma_{m_i}(P_i), \sigma_{m}(Q)\} =\sum_{v=1}^n\frac{\partial \sigma_{m_i}(P_i)}{\partial \xi_v} \partial_v(\sigma_m(Q)).
$$
But $(\sigma_{m_1}(P_1),\ldots ,\sigma_{m_n}(P_n)):\da^n\rightarrow \da^n$ is a finite covering, so $\det (\partial \sigma_{m_i}(P_i)/\partial \xi_j)\neq 0$. Therefore, $\sigma_m(Q)$ must have constant coefficients. 

Now we have
$$
k[\sigma_{m_1}(P_1),\ldots ,\sigma_{m_n}(P_n)]\subset \gr (B)\subset k[\xi_1,\ldots ,\xi_n].
$$
But $k[\xi_1, \ldots ,\xi_n]$ is finitely generated  as $k[\sigma_{m_1}(P_1),\ldots ,\sigma_{m_n}(P_n)]$-module, hence $\gr B$ is finitely generated of Krull dimension $n$. 

It will be useful to introduce the analogue of the Rees ring $\tilde{B}$ constructed by the filtration on the ring $B$: $\tilde{B}=\bigoplus\limits_{n=0}^{\infty} B_n $. The ring $\tilde{B}$ is a subring of the polynomial ring $B[s]$.
 For the fields of fractions we have $\Quot \tilde{B} = \Quot B[s]$.
 Besides, $\gr B=\tilde{B}/(1_1)$, where by $1_1$ we denote the element $1\in B_1$. Using \cite[Ch.III, \S 2.9, Prop.~10]{Bu} one obtains that $B$ is finitely generated as $k$-algebra and the generators of $B$ together with the element $1_1$ generate the algebra $\tilde{B}$. Hence we can compute   the Krull dimension of the ring $B$:
$$
 \dim B=\trdeg \Quot B=\trdeg \Quot \tilde{B}-1=\trdeg \Quot (\tilde{B}/(1_1))=\trdeg \Quot (\gr B)=n  \mbox{,}
$$
since $(1_1)$ is a prime ideal of height $1$ in the ring $\tilde{B}$ by Krull's height theorem.  
\end{proof}

\subsubsection{Case of dimension 2}

From now on we consider a complete $k$-algebra $R=k[[x_1,x_2]]$ with a coordinate system $(x_1,x_2)$. 

\begin{lemma}
\label{lemma5}
Let $P,P_1,Q$ be elements of $D$ of order $m,k,n$ respectively, all with constant principal symbols. Assume $k$ is an algebraically closed field.  
\begin{enumerate}
\item\label{enu1} 
If there exists a point $p\in \Sup \Ch_0(Q)\backslash (\Sup \Ch_0(P)\cup \Sup \Ch_0(P_1))$ which is simple in $\Ch_0(Q)$, then there exists a linear change of coordinates $(x_1,x_2)=(x_1',x_2')(a_{ij})$ 
such that in the new coordinates 
\begin{equation}
\label{enu1.1}
\sigma_m(P)={\xi_2'}^m+ \sum_{q=1}^mh_q{\xi_1'}^q{\xi_2'}^{m-q}, 
\end{equation}
\begin{equation}
\label{enu1.2}
\sigma_k(P_1)=a_0{\xi_2'}^k+ \sum_{q=1}^ka_q{\xi_1'}^q{\xi_2'}^{k-q},
\end{equation}
\begin{equation}
\label{enu1.3} 
\sigma_n(Q)=\xi_1'{\xi_2'}^{n-1}+ \sum_{q=2}^nl_q{\xi_1'}^q{\xi_2'}^{n-q},
\end{equation}
where $h_q,a_q,l_q\in k$, $a_0\neq 0$. 
\item\label{enu2}
If the function $\sigma_n(P)^m/\sigma_m(Q)^n$ is not a constant, then for almost all $\alpha \in k$ the triple $P,P_1, Q_{\alpha}=Q^n+\alpha P^m$ satisfies the assumptions of item \ref{enu1}.
\end{enumerate}
\end{lemma}

\begin{proof} \ref{enu1}. Let $F, F_1, G$ be the principal symbols of $P,P_1,G$ expressed in coordinates $\xi_1,\xi_2$. Then the point $p$ has coordinates say $(a_{21}:a_{22})$ and $F(a_{21},a_{22})F_1(a_{21},a_{22})\neq 0$. We can choose $(a_{21},a_{22})$ such that $F(a_{21},a_{22})=1$. 

We can choose $(a_{11},a_{12})$ such that $\det (a_{ij})\neq 0$ and 
$$
\frac{\partial \sigma}{\partial \xi_1}(a_{21},a_{22})a_{11}+\frac{\partial \sigma}{\partial \xi_2}(a_{21},a_{22})a_{12}=1
$$
(since $(\frac{\partial \sigma}{\partial \xi_1}(a_{21},a_{22}), \frac{\partial \sigma}{\partial \xi_2}(a_{21},a_{22}))\neq (0,0)$ as $(a_{21}=a_{22})$ is a simple root of $G$). 

With the coordinate change 
$$
(x_1,x_2)=(x_1',x_2')
\left (
\begin{array}{cc}
a_{11}&a_{12}\\
a_{21}&a_{22}
\end{array}
\right ),
\mbox{\quad} (\xi_1,\xi_2)=(\xi_1',\xi_2')
\left (
\begin{array}{cc}
a_{11}&a_{12}\\
a_{21}&a_{22}
\end{array}
\right )
$$
we get 
$$
\sigma_m(P)=\tilde{F}(\xi_1',\xi_2')=F(a_{11}\xi_1'+a_{21}\xi_2', a_{12}\xi_1'+a_{22}\xi_2')
$$
(and similar expressions for $\sigma_k(P_1)$, $\sigma_n(Q)$) and $\tilde{F}(0,1)=F(a_{21},a_{22})=1$, $\tilde{F}_1(0,1)=F_1(a_{21},a_{22})\neq 0$, $\tilde{G}(0,1)=0$,
$$
\frac{\partial \tilde{G}}{\partial \xi_1}(0,1)=\frac{\partial G}{\partial \xi_1}(a_{21},a_{22})a_{11}+\frac{\partial G}{\partial \xi_2}(a_{21},a_{22})a_{12}=1.
$$
So, $\sigma_m(P)$ is a monic polynomial in $\xi_2'$, $\sigma_k(P_1)$ is a monic polynomial in $\xi_2'$ up to non-zero factor, and $\sigma_n(Q)=\xi_1'\tilde{H}(\xi_1',\xi_2')$ with $\tilde{H}$ monic in $\xi_2'$.

\ref{enu2}. By hypothesis $F^n/G^m$ is not constant, so if $H=GCD(F^n,G^m)$ and $F^n=F_1H$, $G^m=G_1H$ then $\deg F_1=\deg G_1=N>0$. Since $F_1,G_1$ are coprime, the polynomial $G_1+tF_1\in k[\xi_1,\xi_2,t]$ is irreducible and defines an irreducible curve $C\subset \dpp^1\times \da^1$, and the projection to $\da^1$ defines a finite $N:1$ covering $C\rightarrow \da^1$. 

The fibres $C_{\alpha}$ over $\alpha\in k$ are divisors on $\dpp^1$, which are reduced for $\alpha\in \da^1\backslash S$, $S$ the finite branch locus of $C\rightarrow \da^1$ (cf. \cite[cor. 10.7, ch.III]{Ha}). Also, for $\alpha\neq \beta$, we have $C_{\alpha}\cap C_{\beta}=\emptyset$, since $F_1,G_1$ have no common divisor. 

Hence there is a finite set $T\subset \da^1$ such that for no point $\alpha\in \da^1\backslash T$ $C_{\alpha}$ meets the finite set $\Sup \Ch_0(P)\cup \Sup \Ch_0(P_1)$. So, for $\alpha\in \da^1\backslash (S\cup  T)$ all points of $C_{\alpha}$ have multiplicity one and $C_{\alpha}$ is disjoint to $\Sup (\Ch_0(P))\cup \Sup (\Ch_0(P_1))$. Since $\Sup (\Ch_0(H))\subset \Sup \Ch_0(P)$, $C_{\alpha}$ is also disjoint to $\Sup (\Ch_0(H))$. 

Since $G^m+\alpha F^n=\sigma_{mn}(Q^m+\alpha P^n)= (G_1+\alpha F_1)H$, any point of $C_{\alpha}\subset \Ch_0(Q^m+\alpha P^n)$ satisfies the condition of item \ref{enu1}.  
\end{proof}

\begin{defin}
\label{orter}
For a commutative ring $B$ of operators, $B\subset {D}$, we define numbers $\tilde{N}_B$, $N_B$ as 
$$
\tilde{N}_B=GCD\{\Ord (a), \mbox{\quad} a\in B\},
$$
$$
N_B=GCD\{ q(a),  \mbox{\quad $a\in B$ such that  $\ord_{\Gamma} (a)=(0,q(a))$ and $\Ord (a)=q(a)$}\}.
$$
\end{defin}

\begin{defin}
\label{ggg}
We say that a commutative ring $B\subset {D}$ is strongly admissible if $\tilde{N}_B=N_B$ (cf. also definitions \ref{verygood}, \ref{ggg1}).
\end{defin}

\begin{prop}
\label{chvar}
Let $B$ be a commutative ring of differential operators, $B\subset D$, $k$ is an algebraically closed field, such that $B$ contains two operators $P,Q$ of order $m,n$ with constant principal symbols and such that $\sigma_m(P)^n$/$\sigma_n(Q)^m$ is a non constant function on $\dpp^1$. 

Then there exist a $k$-linear change of coordinates as in lemma \ref{lemma5} such that $N_B=\tilde{N}_B$. 
\end{prop}

\begin{proof} By lemma \ref{lemma5} we can assume without loss of generality that operators $P,Q$ satisfy (\ref{enu1.1}), (\ref{enu1.3}) from the statement of lemma \ref{lemma5}.  Let $X$ be an operator such that $GCD(\Ord (X), \Ord (P))= \tilde{N}_B$. 

By lemma \ref{wellness} the symbol $s_X$ of $X$ is a homogeneous polynomial with constant coefficients. Now by lemma \ref{lemma5} we obtain that there exists $\alpha$ and a change of coordinates such that the symbols $s_{Q_{\alpha}}, s_{P}, s_X$, where $Q_{\alpha}=\alpha Q^n+P^m$, satisfy
$$
s_{P}={\partial_2'}^{\Ord (P)}+\ldots, \mbox{\quad} s_{X}={\partial_2'}^{\Ord (X)}+\ldots, \mbox{\quad} s_{Q_{\alpha}}=\partial_1'{\partial_2'}^{\Ord (Q_{\alpha})-1}+\ldots .
$$
Clearly this is the needed $k$-linear change of variables.  
\end{proof}

\subsubsection{Growth conditions}

In this subsection we give several new definitions and technical statements.  

\begin{defin}
\label{defin3}
We say  that an operator $P\in \hat{E}_+$ has order $\ord_{\Gamma} (P)=(k,l)$ if $P=\sum_{s=-\infty}^lp_s\partial_2^s$, where $p_s\in \hat{D}_1$, $p_l\in k[[x_1,x_2]][\partial_1]=D_1$, and $\Ord (p_l)=k$. 

We say that an operator $P\in \hat{E}_+$, $P=\sum p_{ij}\partial_1^{i}\partial_2^{j}$ with $\ord_{\Gamma} (P)=(k,l)$ {\it satisfies the condition $A_{\alpha}$}, $\alpha\ge 0$ if 
$$
\leqno{(A_{\alpha})}
\mbox{\qquad}
\ord_M(p_{ij})\ge \left \{ 
\begin{array}{cc}
0& \mbox{\quad if  $i\le \alpha (l-j)+k$}\\
i-\alpha (l-j)-k& \mbox{\quad otherwise}
\end{array}
\right.
$$
In this case and if $\alpha \neq 0$ we define its {\it full order} as $ford(P):=k/\alpha +l$.

We will say that an operator $Q\in \hat{E}_+$, $Q=\sum q_{ij}\partial_1^{i}\partial_2^{j}$ {\it satisfies the condition $A_{\alpha}$ for order $(k,l)$} if $A_{\alpha}$ holds for all  $q_{ij}$. 
\end{defin}

\begin{defin}
\label{strong}
We say that an operator $P\in {E}_+$, $P=\sum p_{ij}\partial_1^{i}\partial_2^{j}$ with $\ord_{\Gamma} (P)=(k,l)$ {\it satisfies the strong condition $A_{\alpha}$}, $\alpha\ge 0$ if 
$$
\leqno{(B_{\alpha})}
\mbox{\quad} p_{ij}=0 \mbox{\quad} for \mbox{\quad} i > \alpha (l-j) +k.
$$

We will say that an operator $Q\in \hat{E}_+$, $Q=\sum q_{ij}\partial_1^{i}\partial_2^{j}$ {\it satisfies the strong condition $A_{\alpha}$ for order $(k,l)$} if $B_{\alpha}$ holds for all  $q_{ij}$. 
\end{defin}

\begin{defin}
\label{sstrong}
We say that an operator $P\in {E}_+$, $P=\sum p_{ij}\partial_1^{i}\partial_2^{j}$ with $\ord_{\Gamma} (P)=(k,l)$ {\it satisfies the super strong condition $A_{\alpha}$}, $\alpha\ge 0$ if 
$$
\leqno{(C_{\alpha})}
\mbox{\quad} p_{ij}=0 \mbox{\quad} for \mbox{\quad} i > \alpha (l-j) +k
$$
and the highest coefficient of the differential operator $p_{ij}$ is a constant.  

We will say that an operator $Q\in \hat{E}_+$, $Q=\sum q_{ij}\partial_1^{i}\partial_2^{j}$ {\it satisfies the super strong condition $A_{\alpha}$ for order $(k,l)$} if $C_{\alpha}$ holds for all  $q_{ij}$. 
\end{defin}

\begin{nt}
\label{nt1.9}
Clearly, we have the following implications: $C_{\alpha} \Rightarrow B_{\alpha} \Rightarrow A_{\alpha}$. 
\end{nt}

\begin{nt}
\label{nt2}
It is easy to see that if $P\in\hat{E}_+$  satisfies the condition $A_{\alpha}$ or strong $A_{\alpha}$, then it satisfies the condition $A_{\kappa}$ or strong $A_{\kappa}$ for any $\kappa > \alpha$. 
\end{nt}

\begin{defin}
\label{defin5}
Assume $P\in\hat{D}_1$, $P=\sum p_s\partial_1^s$ is an operator with the following condition:
there exists a number $f(P)$ such that $\ord_M(p_s)\ge s-f(P)$ if $s\ge f(P)$. Then we say that $P$ {\it satisfies the condition } $AA_{f(P)}$. 
\end{defin}

\begin{defin}
\label{defin5.1}
Assume $P\in {D}_1$, $P=\sum_{s\ge 0} p_s\partial_1^s$ is an operator with the following condition:
there exists a number $f(P)$ such that $p_s=0$ if $s >f(P)$. Then we say that $P$ {\it satisfies the strong condition } $AA_{f(P)}$ (or $BB_{f(P)}$). 
\end{defin}

\begin{defin}
\label{defin5.2}
Assume $P\in {D}_1$, $P=\sum_{s\ge 0} p_s\partial_1^s$ is an operator with the following condition:
there exists a number $f(P)$ such that $p_s=0$ if $s >f(P)$ and $p_{f(P)}\in k$. Then we say that $P$ {\it satisfies the super strong condition } $AA_{f(P)}$ (or $CC_{f(P)}$). 
\end{defin}

\begin{nt}
\label{nt3}
It is easy to see that if $P\in\hat{D}_1$  satisfies the condition $AA_{\kappa}$ or the (super) strong $AA_{\kappa}$, then it satisfies the condition $AA_{\kappa'}$ or the (super) strong $AA_{\kappa'}$ for any $\kappa'> \kappa$. 
\end{nt}

\begin{nt}
\label{nt4}
Note that $P\in \hat{E}_+$, $P=\sum p_s\partial_2^s$ satisfies $A_{\alpha}$ or (super) strong $A_{\alpha}$ if and only if its coefficients $p_s$ satisfy the conditions $AA_{\alpha (ford (P)-s)}$ or (super) strong $AA_{\alpha (ford (P)-s)}$ correspondingly. 

Analogously, $P$ satisfies $A_{\alpha}$ for $(k,l)$ or (super) strong $A_{\alpha}$ for $(k,l)$ if and only if its coefficients $p_s$ satisfy the conditions $AA_{\alpha (l-s)+k}$ or (super) strong $AA_{\alpha (l-s)+k}$. 

Note also that if $P$ satisfies $A_{\alpha}$ for $(k,l)$ then it satisfies $A_{\alpha}$ for any pair $(k_1,l_1)$ such that $l_1+k_1/\alpha =l+k/\alpha$. The same is true for (super) strong conditions.
\end{nt}

\begin{lemma}
\label{lemma10}
Assume $P_1,P_2\in \hat{D}_1$ satisfy the conditions $AA_{f(P_1)}, AA_{f(P_1)}$ correspondingly. Then $P_1P_2$ is an operator satisfying the condition $AA_{f(P_1)+f(P_2)}$. 

The same assertion is true for $P_1,P_2\in {D}_1$ satisfying strong or super strong conditions. 
\end{lemma}

\begin{proof} It suffices to prove lemma for $P_1=p_i\partial_1^{i}$. Let $P_2=\sum p_{2,j}\partial_1^j$ and $P_1P_2=\sum_{k=0}^{\infty}x_k\partial_1^k$. We have
$$
P_1P_2=\sum_{j=0}^{i}p_iC_i^j\partial_1^{j}(P_2)\partial_1^{i-j}
$$
whence 
$$\ord_M(x_{f(P_1)+f(P_2)+l})\ge \min_{j}\{\ord_M(p_i)+\ord_M(p_{2,f(P_1)+f(P_2)+l+j-i})\}.$$ 

If $i\le f(P_1)$, then $f(P_1)+f(P_2)+l+j-i\ge f(P_2)+l$, whence 
$$\ord_M(p_i)+\ord_M(p_{2,f(P_1)+f(P_2)+l+j-i})\ge l$$ 
for any $j$. 

If $i>f(P_1)$, then 
$$\ord_M(p_i)+\ord_M(p_{2,f(P_1)+f(P_2)+l+j-i})\ge i-f(P_1)+f(P_1)+l+j-i\ge l$$ 
for any $j$. So, $\ord_M(x_{f(P_1)+f(P_2)+l})\ge l$. 

The statement for (super) strong conditions is obvious. 
\end{proof}

\begin{lemma}
\label{lemma9}
Assume $P_1,P_2\in \hat{E}_+$ satisfy the condition $A_{\alpha}$ with $\alpha \ge 1$ for $(k_1, l_1)$ and $(k_2,l_2)$ respectively. 
Then $P_1P_2$ satisfies the condition $A_{\alpha}$ for $(k_1+k_2,l_1+l_2)$. 

In particular, if $P_1,P_2$ satisfy the condition $A_{\alpha}$ with $\alpha \ge 1$, then $P_1P_2$ satisfies the condition $A_{\alpha}$ 
and $\ord_{\Gamma}(P_1P_2)= \ord_{\Gamma}(P_1)+\ord_{\Gamma}(P_2)$. 

The same assertions are true for $P_1,P_2\in E_+$ satisfying (super) strong conditions.
\end{lemma}

\begin{proof} We'll prove the assertions in (super) strong and not in strong cases simultaneously. 

It  suffices to prove lemma for the product of two summands of $P_1, P_2$, say $p_k\partial_2^k$, $p_l\partial_2^l$, since any summand in $P_i$ satisfies $A_{\alpha}$ for $(k_i,l_i)$, $i=1,2$.  
We have 
\begin{equation}
\label{vs}
(p_k\partial_2^k)(p_l\partial_2^l)=\sum_{j=0}^{\infty}C_k^jp_k\partial_2^j(p_l)\partial_2^{k+l-j}.
\end{equation}
Note that $p_k$ satisfies the condition $AA_{f(p_k)}$, where $f(p_k)=\alpha (l_1-k)+k_1$, 
$p_l$ satisfies the condition $AA_{f(p_l)}$, where $f(p_l)=\alpha (l_2-l)+k_2$. Note also that $\partial_2^j(p_l)$ satisfies the condition $AA_{f(p_l)}$ in the (super) strong case and it satisfies the condition $AA_{f(p_l)+j}$ not in the strong case. So, by lemma \ref{lemma10} we have $f(p_k\partial_2^j(p_l))=f(p_k)+f(\partial_2^j(p_l))\le \alpha (l_1+l_2-(k+l-j))+k_1+k_2$, whence each summand of (\ref{vs}) satisfies the condition $A_{\alpha}$ in definition \ref{defin3} for $(k_1+k_2,l_1+l_2)$. Hence, the same is true for $P_1P_2$. 

Clearly, $\ord_{\Gamma}(P_1P_2)= \ord_{\Gamma}(P_1)+\ord_{\Gamma}(P_2)$. If $P_i$ satisfy $A_{\alpha}$, then they satisfy $A_{\alpha}$ for $\ord_{\Gamma} (P_i)$. Therefore, $P_1P_2$ satisfies $A_{\alpha}$ for $\ord_{\Gamma}(P_1P_2)$, i.e. $P_1P_2$ satisfies $A_{\alpha}$.  
\end{proof}

\begin{corol}
\label{corol1}
If the operator $S=1-S^-$, where $S^-\in \hat{D}_1[[\partial_2^{-1}]]\partial_2^{-1}$, satisfies the condition $A_{\alpha}$ or (super) strong $A_{\alpha}$ with $\alpha \ge 1$, then the operator $S^{-1}$ also satisfies it. 
\end{corol}

\begin{proof} It follows from the proof of lemma \ref{lemma9}, since $\ord_{\Gamma} (S)=(0,0)$ and $S^{-1}=1+\sum_{q=1}^{\infty}(S^-)^{q}$.  
\end{proof}

\begin{corol}
\label{corol1.5}
Consider the set 
$$
\Pi_{\alpha}=\{P\in \hat{E}_+| \mbox{\quad there exist  $(k,l)\in \dz_+\oplus \dz$ such that $P$ satisfies $A_{\alpha}$ for $(k,l)$}\}\subset \hat{E}_+.
$$
It is an associative subring with unity.
\end{corol}

\begin{proof} Take $P_1,P_2\in \Pi_{\alpha}$. By lemma \ref{lemma9}, we have $P_1P_2\in \Pi_{\alpha}$. We also have $P_1+P_2\in \Pi_{\alpha}$, because $P_1+P_2$ satisfies  $A_{\alpha}$ for those pair $(k_i,l_i)$, $i=1,2$, where the value of $l_i+k_i/\alpha$ is greater (cf. also remark \ref{nt4}). So, $\Pi_{\alpha}$ is an associative subring of $\hat{E}_+$ with unity 1.  
\end{proof}  

\begin{lemma}
\label{lemma6}
Let $P,Q\in \hat{D}\subset \hat{E}_+$ be commuting monic operators such that $\ord_{\Gamma} (P)=(0,k)$, $\ord_{\Gamma} (Q)= (1,l)$. Then 
\begin{enumerate}
\item\label{a}
There exist unique operators $L_1\in \hat{E}_+$, $L_2\in \hat{E}_+$ such that $L_2^k=P$, $L_1L_2^l=Q$, $[L_1,L_2]=0$. 
\item\label{b}
If $P,Q$ satisfy the condition $A_{\alpha}$ with $\alpha \ge 1$  then $L_1,L_2$ satisfy the condition $A_{\alpha}$.
\item\label{c}
If $P,Q\in D$ then $L_1,L_2\in \hat{E}_+\cap E$.
\item\label{d}
If $P,Q\in D$ satisfy the (super) strong condition $A_{\alpha}$ with $\alpha \ge 1$  then $L_1,L_2$ satisfy the (super) strong condition $A_{\alpha}$.
\end{enumerate} 
\end{lemma}

\begin{proof} \ref{a}. We can find each coefficient of the operator $L_2=\partial_2+u_0+u_{-1}\partial_2^{-1}+\ldots $ step by step, by solving the system of equations, which can be obtained by comparing the coefficients of $P$ and $L_2^k$:
\begin{equation}
\label{vs5}
ku_0=p_{k-1}, \mbox{\quad} ku_{-i}+F(u_0,\ldots ,u_{-i+1})=p_{k-1-i},
\end{equation}
where $F$ is a polynomial in $u_0,\ldots ,u_{-i+1}$ and their derivatives. 
Clearly, this system is uniquely solvable. So, the operator $L_2$ is uniquely defined. Note that $L_2$ is invertible element, $L_2^{-1}\in \hat{E}_+$ and $\ord_{\Gamma} (L_2^{-1})=(0,-1)$. Therefore, $L_1=QL_2^{-l}$ is also uniquely defined. 

The same arguments show that item \ref{c} is true. 

\ref{b} and \ref{d}. We'll prove the assertions in (super) strong and not in strong cases simultaneously.

It follows from (\ref{vs5}) that $u_0$ satisfies $A_{\alpha}$ for $\ord_{\Gamma} (L_2)$ or, equivalently, by remark \ref{nt4},  $u_0$ satisfies $AA_{\alpha}$.  Assume that $F(u_0,\ldots ,u_{-i+1})$ in (\ref{vs5}) satisfies $AA_{\alpha (1+i)}$. Then by (\ref{vs5}) $u_{-i}$ will also satisfy $AA_{\alpha (1+i)}$. Let's show that $F(u_0,\ldots ,u_{-i})$ satisfies $AA_{\alpha (2+i)}$. 

We have 
$$
L_2^k=(\partial_2+u_0+ \ldots +u_{-i}\partial_2^{-i})^k+ u_{-i-1}\partial_2^{-i-2+k}+ \mbox{higher order terms}. 
$$
By lemma \ref{lemma9} and remark \ref{nt4} the operator $(\partial_2+u_0+ \ldots +u_{-i}\partial_2^{-i})^k$ satisfies $A_{\alpha}$. But $F(u_0,\ldots ,u_{-i})$ is a coefficient at $\partial_2^{-i-2+k}$ of this operator. So, it satisfies $AA_{\alpha (2+i)}$ by remark \ref{nt4}. 

Now by induction we obtain item \ref{b} and \ref{d} for $L_2$. The operator $L_1$ satisfies $A_{\alpha}$ by lemma \ref{lemma9} and corollary \ref{corol1}.  
\end{proof}

\subsubsection{Quasi elliptic rings of commuting operators}

Motivated by this lemma and by lemma \ref{lemma5} we'll give the following definitions:
\begin{defin}
\label{elliptic}
The ring $B\subset \hat{E}_+$ of commuting operators is called quasi elliptic if it contains two monic operators $P,Q$ such that $\ord_{\Gamma} (P)= (0,k)$ (see definition \ref{defin3}) and $\ord_{\Gamma} (L)=(1,l)$ for some $k,l\in \dz$.  

The ring $B$ is called $\alpha$-quasi elliptic if $P,Q$ satisfy the condition $A_{\alpha}$. 
\end{defin}

\begin{defin}
\label{defin4,5}
We say that commuting monic operators $P,Q\in \hat{E}_+$ with $\ord_{\Gamma} (P)=(0,k)$, $\ord_{\Gamma} (Q)= (1,l)$  are almost normalized if 
$$P=\partial_2^k+ \sum_{s=-\infty}^{k-1}p_{s}\partial_2^{s} \mbox{\quad} Q=\partial_1\partial_2^l+ \sum_{s=-\infty}^{l-1}q_{s}\partial_2^{s},$$ where $p_s,q_s\in \hat{D}_1$.  

We say that $P,Q$ are  normalized if 
$$P=\partial_2^k+ \sum_{s=-\infty}^{k-2}p_{s}\partial_2^{s} \mbox{\quad} Q=\partial_1\partial_2^l+ \sum_{s=-\infty}^{l-1}q_{s}\partial_2^{s},$$ where $p_s,q_s\in \hat{D}_1$.  
\end{defin}

\begin{lemma}
\label{lemma7}
For any two commuting monic operators $P,Q\in \hat{D}$ with $\ord_{\Gamma} (P)=(0,k)$, $\ord_{\Gamma} (Q)= (1,l)$ we have
\begin{enumerate}
\item\label{1*}
\begin{enumerate}
\item\label{1a}
There exists an invertible function $f\in k[[x_1,x_2]]$ such that the operators $f^{-1}Pf, f^{-1}Qf$ will be almost normalized.
\item\label{1b}
There exists an operator $S=f+S^-$, where $S^-\in \hat{D}_1\partial_1\subset \hat{E}_+$ and invertible $f\in k[[x_1,x_2]]$, such that the operators $S^{-1}PS, S^{-1}QS$ will be normalized. 
\item\label{2*}
If $S_1$ is another operator with such a property, then $S^{-1}S_1\in k$. 
\end{enumerate}
\item\label{3*}
\begin{enumerate}
\item\label{3a}
If $P,Q$ satisfy the condition $A_{\alpha}$, then the almost normalized operators in \ref{1a} also satisfy $A_{\alpha}$. 
\item\label{3b}
If $P,Q$ satisfy the condition $A_{\alpha}$ with $\alpha = 1$, then  $S$ in \ref{1b} satisfies the condition $A_{\alpha }$. 
In this case the normalized operators in \ref{1b} also satisfy $A_{\alpha }$. 
\end{enumerate}
\end{enumerate}
\end{lemma}

\begin{proof} First let's show that there exists a function $f\in k[[x_1,x_2]]^*$ such that 
\begin{equation}
\label{vs2}
f^{-1}Pf=\partial_2^k+ \sum_{s=0}^{k-1}p'_{s}\partial_2^{s}, \mbox{\quad} f^{-1}Qf=\partial_1\partial_2^l+ \sum_{s=0}^{l-1}q'_{s}\partial_2^{s}.
\end{equation}
Let $Q=\sum_{s=0}^{l}q_{s}\partial_2^{s}$ and $q_l=\partial_1\partial_2^l+g$. Then easy direct computations show that for any function $f\in k[[x_1,x_2]]^*$ we have 
$$
f^{-1}Pf=\partial_2^k+ \sum_{s=0}^{k-1}p'_{s}\partial_2^{s}, \mbox{\quad} f^{-1}Qf=\partial_2^l(\partial_1+ f^{-1}\partial_1(f)+g)+ \sum_{s=0}^{l-1}q'_{s}\partial_2^{s}
$$
with some coefficients $p'_s,q'_s\in \hat{D}_1$. 
Hence, we can find a needed function in the form $f=\exp (-\int gdx_1)$. 

So, we have reduced the problem to the operators $P,Q$ that look like the right hand side in (\ref{vs2}). 
Analogously, we can find a function $f\in k[[x_2]]^*$ such that, starting with the operators $P,Q$ that look like the right hand side in (\ref{vs2}), we'll have 
\begin{equation}
\label{vs3}
f^{-1}Pf=\partial_2^k+ \sum_{s=0}^{k-1}p'_{s}\partial_2^{s}, \mbox{\quad} f^{-1}Qf=\partial_1\partial_2^l+ \sum_{s=0}^{l-1}q'_{s}\partial_2^{s}, 
\end{equation}
where the element $p'_{k-1}$ has no free term. 
Again, direct computations show that for any function $f\in k[[x_2]]^*$ we have 
$$
f^{-1}Pf=\partial_2^k+ \sum_{s=0}^{k-1}p'_{s}\partial_2^{s}, \mbox{\quad} f^{-1}Qf=\partial_2^l(\partial_1+ f^{-1}\partial_1(f)+g)+ \sum_{s=0}^{l-1}q'_{s}\partial_2^{s},
$$
where $p'_{k-1}=p_{k-1}+ kf^{-1}\partial_2(f)$ (note that $f$ commutes with $p_s$). Since $[P,Q]=0$, we must have $\partial_1(p_{k-1})=0$. Hence, we can find a needed function $f\in k[[x_2]]^*$. 

Note that any function $f\in k[[x_1,x_2]]^*$ that preserves two operators of the form (\ref{vs3}) must be a constant. It follows immediately from the formulae above. 

So, we have reduced the problem to the operators $P,Q$ that look like the right hand side in (\ref{vs3}). Let's show that there exists an operator $S= 1+S^-$, $S^-\in \hat{D}_1\partial_1$ such that 
\begin{equation}
\label{vs4}
S^{-1}PS=\partial_2^k+ \sum_{s=0}^{k-2}p'_{s}\partial_2^{s}, \mbox{\quad} S^{-1}QS=\partial_1\partial_2^l+ \sum_{s=0}^{l-1}q'_{s}\partial_2^{s}. 
\end{equation}
Since $\partial_1(p_{k-1})=0$, we may look for an operator $S$ such that $\partial_1(S)=0$. Direct computations (note that $S$ commutes with $p_{k-1}$) show that for such an operator we have
$$
S^{-1}PS=\partial_2^k+ (p_{k-1} + kS^{-1}\partial_2(S))\partial_2^{k-1}+ \sum_{s=0}^{k-2}p'_{s}\partial_2^{s}, \mbox{\quad} S^{-1}QS=\partial_1\partial_2^l+ \sum_{s=0}^{l-1}q'_{s}\partial_2^{s}. 
$$
Hence, we can find a needed operator in the form $S=\exp (-\int p_{k-1}/kdx_2)$. Since $p_{k-1}$ has no free term, $\partial_1(p_{k-1})=0$, and there is $(-\int p_{k-1}/kdx_2)$ with $\ord_{M_2}(-\int p_{k-1}/kdx_2)>0$, this exponent is well defined, and $S\in \hat{D}_1$. 

Note that an operator $S$ that preserves normalized operators $P,Q$ must be an operator with constant coefficients. It follows easily from the calculations above. Since it is invertible, it must be a constant. 
Summing all together, we obtain the proof of items \ref{1*}, \ref{2*}.

The proof of \ref{3a} follows immediately from lemma \ref{lemma9}. 

To prove \ref{3b} let's note that, by remark \ref{nt4}, the coefficient $p_{k-1}$ satisfies $AA_{\alpha}$.  Hence, $(-\int p_{k-1}/kdx_2)$ above satisfies $AA_{\alpha -1}$. Since in our case $\alpha =1$, we obtain that $S$ satisfies $AA_{0}$ as a sum of operators satisfying $AA_{0}$, because $(-\int p_{k-1}/kdx_2)^s$ satisfies $AA_{0}$ by lemma \ref{lemma10}. It follows then that $S$ satisfies $A_{\alpha}$. 
The rest of the proof follows from lemma \ref{lemma9} and corollary \ref{corol1}. 
\end{proof} 

\begin{lemma}
\label{lemma8}
Let $L_1, L_2\in \hat{E}_+$ be commuting monic almost normalized operators with $\ord_{\Gamma} (L_2)=(0,1)$, $\ord_{\Gamma} (L_1)=(1,0)$:
$$L_1=\partial_1 + \sum_{q=1}^{\infty}v_q\partial_2^{-q}, \mbox{\quad } L_2= \partial_2 + \sum_{q=0}^{\infty}u_q\partial_2^{-q}.$$
 Then 
\begin{enumerate}
\item
\begin{enumerate}
\item\label{1}
There exists an operator $S=1+S^-$, where $S^-\in \hat{D}_1[[\partial_2^{-1}]]\partial_2^{-1}$ such that $S^{-1}\partial_1S=L_1$, $S^{-1}L_{20}S=L_2$, where $L_{20}=\partial_2+u_0$. 
\item\label{2}
If $S_1$ is another operator with such a property, then $S^{-1}S_1\in k [\partial_1]((L_{20}^{-1}))$. 
\end{enumerate}
\item\label{3}
If $L_1, L_2\in \hat{E}_+\cap E$, then $S\in \hat{E}_+\cap E$.
\item\label{4}
\begin{enumerate}
\item
If $L_1,L_2$ satisfy the condition $A_{\alpha}$ with $\alpha \ge 1$, then there exists $S$ satisfying the condition $A_{2\alpha -1}$; in particular, if $\alpha =1$, then $S$ satisfies $A_{\alpha}$. 
\item 
If $S_1$ is another operator with such a property, then $S^{-1}S_1\in k [\partial_1]((L_{20}^{-1}))$ and satisfies $A_{2\alpha -1}$. 
\end{enumerate}
\end{enumerate}
\end{lemma}

\begin{proof} \ref{1}. It suffices to prove the following fact: if 
$$L_1=\partial_1 + \sum_{q=k}^{\infty}v_q\partial_2^{-q}, \mbox{\quad } L_2= \partial_2 + u_0+ \sum_{q=k}^{\infty}u_q\partial_2^{-q},\mbox{\quad } [L_1,L_2]=0,$$ then there exists an operator $S_k=1+s_k\partial_2^{-k}$ such that 
$$S_k^{-1}L_1S_k=\partial_1 + \sum_{q=k+1}^{\infty}v'_q\partial_2^{-q},\mbox{\quad } S_k^{-1}L_2S_k=\partial_2 +u_0+ \sum_{q=k+1}^{\infty}u'_q\partial_2^{-q}.$$

Indeed, if this fact is proved, then $S^{-1}=\prod_{q=1}^{\infty}S_k$, where $S_1$ is taken for given $L_1,L_2$, $S_2$ is taken for $S_1^{-1}L_1S_1, S_1^{-1}L_2S_1$, and so on.  

To prove the fact let's note first that, since $[L_1,L_2]=0$, it follows $\partial_2(v_k)- \partial_1(u_k)+[u_0,v_k]=0$ and $\partial_1(u_0)=0$. After that,
$$
S_k^{-1}\partial_1S_k=\partial_1 +S_k^{-1}\partial_1(S_k )=\partial_1 + \partial_1(s_k)\partial_2^{-k}+\ldots , 
$$
$$ 
S_k^{-1}L_{20}S_k=\partial_2 +S_k^{-1}\partial_2(S_k )+S_k^{-1}u_{0}S_k =\partial_2 + (\partial_2(s_k)+[u_0,s_k])\partial_2^{-k}+\ldots ,
$$
whence $s_k$ can be found from the following system:
\begin{equation}
\label{sv2}
\partial_1(s_k)=-v_k \mbox{\quad} \partial_2(s_k)+[u_0,s_k]=-u_k. 
\end{equation}
This system is solvable, because  $\partial_2(v_k)- \partial_1(u_k)+[u_0,v_k]=0$ and $\partial_1(u_0)=0$ and all coefficients of $u_k,v_k$ belong to $k[[x_1,x_2]]$. 

\ref{2}. If $S_1$ is another operator with such a property, then we must have $[S^{-1}S_1,\partial_1]=0$, $[S^{-1}S_1,L_{20}]=0$.
Note that any element in $\hat{E}_+$ can be rewritten as a series in the ring $\hat{D}_1((L_{20}^{-1}))$. So, we'll assume that $S^{-1}S_1$ is rewritten in such a way. Since $[\partial_1, L_{20}]=0$, the first condition gives $\partial_1(S^{-1}S_1)=0$, i.e. the coefficients of $S^{-1}S_1$ don't depend on $x_1$. 

Now let $S^{-1}S_1=\sum_{q=0}^{\infty} s_qL_{20}^{-q}$ and assume that $s_k$ is a first coefficient such that $[s_k,L_{20}]\neq 0$. Then we have
$$
0=[S^{-1}S_1, L_{20}]=[s_k,L_{20}]L_{20}^{-k}+ \mbox{higher order terms},
$$
whence $[s_k,L_{20}]=0$, a contradiction. But $[s_k,L_{20}]=-\partial_2(s_k)$, because $\partial_1(s_k)=0$ and therefore $[s_k,u_0]=0$. So, we obtain that the  coefficients of $S^{-1}S_1$ don't depend on $x_2$. 

This means that the coefficients of $S^{-1}S_1$ must belong to $k$. Then from definition of the ring $\hat{E}_+$ it follows that $S^{-1}S_1\in k [\partial_1]((L_{20}^{-1}))$.

\ref{3}. The proof is the same as in \ref{1}.

\ref{4}. By corollary \ref{corol1}, the proof of item \ref{4} will follow from the proof of item \ref{1}, if we show that the operators $S_k$ satisfy the condition $A_{2\alpha -1}$. To prove this, we need to show that there is a solution $s_k$ of (\ref{sv2}) satisfying the condition $AA_{(2\alpha -1)k}$. But each solution of (\ref{sv2}) can be written in the form 
\begin{equation}
\label{s_k}
s_k=-\int v_kdx_1 +\int (\int \partial_2(v_k)dx_1 -u_k+[u_0,\int v_kdx_1])dx_2.
\end{equation}
We know that $u_k$ satisfy the condition $AA_{\alpha (1+k)}$ and $v_k$ satisfy the condition $AA_{\alpha k+1}$. So, there is integral $\int v_kdx_1$ satisfying $AA_{\alpha k}$. Then by lemma \ref{lemma10} $[u_0,\int v_kdx_1]$ satisfies $AA_{\alpha (k+1)}$. The term $\int \partial_2(v_k)dx_1$ will satisfy again $AA_{\alpha k+1}$. Since $\alpha (k+1)\ge \alpha k+1$, we obtain that the term $(\int \partial_2(v_k)dx_1 -u_k+[u_0,\int v_kdx_1])$ will satisfy $AA_{\alpha (k+1)}$. Then there is an integral $\int (\int \partial_2(v_k)dx_1 -u_k+[u_0,\int v_kdx_1])dx_2$ satisfying $AA_{\alpha (1+k)-1}$. Since $\alpha (1+k)-1\ge \alpha k$, we obtain that $s_k$ will satisfy  $AA_{\alpha (1+k)-1}$. But $(2\alpha -1)k\ge \alpha (1+k)-1$, hence there is $s_k$ satisfying $AA_{(2\alpha -1)k}$. 
\end{proof}

\section{Classification of subrings of commuting operators}

\subsection{Classification in terms of Schur pairs}
\label{reduction}

Now we are ready to describe a classification of certain rings of commuting operators. In fact, we can do it for all  $1$-quasi elliptic rings (see below). Let's show that many usual rings of commuting differential operators become $1$-quasi elliptic after a change of coordinates. 

Namely, consider a ring $B$ of commuting differential operators that contains two operators $P,Q$ with constant principal symbols satisfying the assumptions of proposition \ref{chvar}. The operators $P,Q$ satisfy the condition $A_{1}$ for order $(k,l)$ and order $(n,m)$ correspondingly, where $k+l= \Ord (P)$, 
$n+m= \Ord (Q)$. By lemma \ref{lemma5} we can find in $B$ (after an appropriate change of variables) two operators $P,Q$ of special type described in this lemma (we use here the same notation for $P,Q$ to point out that these operators satisfy conditions \ref{enu1.1} and \ref{enu1.3} of lemma \ref{lemma5}; we hope this will not lead to a confusion). In particular they satisfy the condition $A_{1}$, and the ring $B$ (after an appropriate change of variables) becomes $1$-quasi elliptic. Moreover, applying proposition \ref{chvar} we see that $B$ (after an appropriate change of variables) becomes strongly admissible. 

\bigskip

Consider now a $1$-quasi elliptic ring of commuting operators $B\subset \hat{D}$ (see definition \ref{elliptic}), and let $P,Q$ be monic operators from $B$ with $\ord_{\Gamma} (P)=(0,k)$, $\ord_{\Gamma} (Q)=(1,l)$. 
By lemma \ref{lemma6}, there exist unique operators $L_1,L_2$ such that $L_2^{k}=P$, $L_1L_2^{l-1}=Q$, and these operators satisfy the condition $A_{1}$.

By lemma \ref{lemma7}, \ref{3b} we can 
assume that they are normalized. Then by lemma \ref{lemma8}, there is an operator $S$ satisfying $A_{1}$, and $SL_1S^{-1}=\partial_1$, $SL_2S^{-1}=\partial_2$.

\begin{lemma}
\label{ochevidn}
Let $X$ be an operator commuting with $P,Q$. Then it commutes also with $L_1, L_2$. 
\end{lemma}

\begin{proof} We have 
$$
0=[P,X]=\sum_{q=0}^{k-1}L_2^q[L_2,X]L_2^{k-1-q},
$$
and $\HT (L_2^q)=\partial_2^q$. If $[L_2,X]\neq 0$, then $\HT ([L_2,X])\neq 0$ (here it suffice to consider the highest term of an operator in $\hat{D}_1((\partial_2^{-1}))=\hat{E}_+$ with respect to $\partial_2$), whence $\HT ([P,X])=k\HT ([L_2,X])\partial_2^{k-1}\neq 0$, a contradiction. So, $[L_2,X]=0$. Then also $[L_1,X]=0$, because $0=[Q,X]=[L_1,X]L_2^{l-1}$.  
\end{proof}
\begin{corol}{(cf. prop. \ref{techn5.2})}
\label{corol3.1}
The set of commuting with $P,Q$ operators is a commutative ring. Moreover, all these operators belong to the ring $\Pi_1$ (see corollary \ref{corol1.5}). 
\end{corol}

\begin{proof} Indeed, if $X$ commutes with $P,Q$, then it commutes with $L_1, L_2$ and therefore $SXS^{-1}$ commutes with $\partial_1,\partial_2$, where from $SXS^{-1}$ is an operator with constant coefficients. Therefore, any two operators commuting with $P,Q$ must commute with each other.  

To prove the second statement consider the space $W_0S^{-1}$, where $W_0=\langle z_1^{-i}z_2^{-j}\mbox{\quad }| i,j\ge 0 \rangle$. Since $S$ satisfies $A_{1}$, we have by corollary \ref{corol1} that $S^{-1}$ satisfies $A_{1}$, and by definition of the action, that the element $z_1^{-k}z_2^{-l}S^{-1}$ also satisfies  $A_{1}$ for any $k,l\ge 0$. Note also that $(W_0S^{-1})(SXS^{-1})\subset (W_0S^{-1})$. Since $\Sup (W_0S^{-1})=\Sup (W_0)$, there is the unique basis $\{w_{i,j}, i,j\ge 0\}$ in $W_0S^{-1}$ with the property $w_{i,j}=z_1^{-i}z_2^{-j}+w_{i,j}^-$, where $w_{i,j}^-\in k[z_1^{-1}][[z_2]]z_2$ and all elements $w_{i,j}$ satisfy the condition $A_{1}$. Therefore the operator $w_{0,0}(SXS^{-1})$ is a finite sum of $w_{i,j}$. So, it belongs to $\Pi_1$ (cf. the proof of corollary \ref{corol1.5}) and therefore $SXS^{-1}\in \Pi_1$ by lemma \ref{lemma9}. 
\end{proof}

So, starting from a $1$-quasi elliptic ring $B$ we obtain a ring of operators with constant coefficients $A=SBS^{-1}\in \Pi_1$ and the space $W=W_0S^{-1}$, $WA\subset W$, with special property. The converse is also true:

\begin{theo}
\label{theo2}
Let $W$ be a $k$-subspace 
$W\subset  k[z_1^{-1}]((z_2))$ 
with $\Sup (W)=W_0$.
Let $\{w_{i,j}, i,j\ge 0\}$ be the unique basis in $W$ with the property $w_{i,j}=z_1^{-i}z_2^{-j}+w_{i,j}^-$, where $w_{i,j}^-\in k[z_1^{-1}][[z_2]]z_2$. Assume that all elements $w_{i,j}$ satisfy the condition $A_{\alpha}$ with $\alpha\ge 1$. 

Then there exists a unique operator $S=1+S^-$ satisfying $A_{\alpha}$, where $S^-\in \hat{D}_1[[\partial_2^{-1}]]\partial_2^{-1}$, such that 
$W_0S=W $. 
\end{theo}

\begin{proof} We can repeat the proof of theorem \ref{theo1} to show that in our situation $S$ satisfies $A_{\alpha}$. Note that $S$ satisfies $A_{\alpha}$ if every $(k,l)$-slice satisfies $A_{\alpha}$ for $(k,l)$. 

To show this we use induction on $(k,l)$. The $(0,0)$-slice is equal $w_{0,0}$, therefore it satisfies $A_{\alpha}$ for $(0,0)$. Assume that each $(p,q)$-slice with $p\le k$, $q\le l$ and $(p,q)\neq (k,l)$ satisfies $A_{\alpha}$ for $(p,q)$. Then from formula \ref{eq1} follows that the $(k,l)$-slice satisfies $A_{\alpha}$ for $(k,l)$, because each element $w_{i,j}$ satisfies $A_{\alpha}$ (cf. corollary \ref{corol1.5}). 
\end{proof}

\begin{corol}
\label{corol2}
Let $W$ be a subspace as in theorem. Let $A\subset  k[z_1^{-1}]((z_2))$ be a ring such that $W A\subset W$. Then we have an embedding $SAS^{-1}\subset \hat{D}$ (here we identify the ring $k[z_1^{-1}]((z_2))$ and $k[\partial_1]((\partial_2^{-1}))$, see definition \ref{action}). 
\end{corol}

\begin{proof} Clearly, $W_0 SAS^{-1}\subset W_0$. Then by proposition \ref{prop2} $SAS^{-1}\in \hat{D}$.  
\end{proof}

Motivated by theorem \ref{theo2} and lemma \ref{lemma8} we'll give the following definitions:

\begin{defin}
\label{1space}
The subspace $W\subset  k[z_1^{-1}]((z_2))$ is called $\alpha$-space, if there exists a basis $w_i$ in $W$ such that $w_i$ satisfy the condition $A_{\alpha}$ for all $i$.
\end{defin}

\begin{defin}
\label{sch}
We say that a pair of subspaces $(A,W)$, where $A,W \subset  k[z_1^{-1}]((z_2))$ and $A$ is a $k$-algebra with unity such that $WA\subset W$, is a $\alpha$-Schur pair if $A\subset  \Pi_{\alpha}$ (see corol. \ref{corol1.5}) and $W$ is a $\alpha$-space. 

We say that $\alpha$-Schur pair is a $\alpha$-quasi elliptic Schur pair if $A$ is a $\alpha$-quasi elliptic ring (see def. \ref{elliptic}; we identify here the ring $k[z_1^{-1}]((z_2))$ with the ring $k[\partial_1]((\partial_2^{-1}))$ via $z_1 \mapsto \partial_1^{-1}$, $z_2\mapsto \partial_2^{-1}$). 
\end{defin}

\begin{defin}{(cf. \cite[def.1]{Zhe})}
\label{admissible}
An operator $T\in \hat{E}_+$ is said to be admissible if it is an invertible operator of order zero such that $T\partial_1 T^{-1}$, $T\partial_2T^{-1}\in k[\partial_1]((\partial_2^{-1}))$. The set of all admissible operators is denoted by $\Adm$ (for a classification of admissible operators see \cite[lemma 7]{Zhe}). 

An operator $T\in \hat{E}_+$ is said to be $\alpha$-admissible if it is admissible and 
satisfies the condition $A_{\alpha}$ (in this case by lemma \ref{lemma9} we have $T\partial_1 T^{-1}$, $T\partial_2T^{-1}\in \Pi_{\alpha}$). The set of all $\alpha$-admissible operators is denoted by $\Adm_{\alpha }$.

We say that two $\alpha$-Schur pairs $(A,W)$ and $(A',W')$ are equivalent if $A'=T^{-1}AT$ and $W'=WT$, where $T$ is an admissible operator. 
\end{defin}

\begin{defin}
\label{rings}
The commutative $\alpha$-quasi elliptic rings $B_1$, $B_2\subset \hat{D}$ are said to be equivalent if there is an invertible operator $S\in \hat{D}_1$ as in lemma \ref{lemma7} \ref{1b} such that $B_1=SB_2S^{-1}$. 
\end{defin}

Summing the arguments above together, we obtain:
\begin{theo}
\label{schurpair} 
There is a one to one correspondence between the classes of equivalent $1$-quasi elliptic Schur pairs $(A,W)$ from definition \ref{admissible} with $\Sup(W)=\langle z_1^{-i}z_2^{-j}\mbox{\quad }| i,j\ge 0 \rangle$ and the classes of equivalent  $1$-quasi elliptic rings (see definitions \ref{elliptic}, \ref{rings}) of commuting operators $B\subset \hat{D}$.    
\end{theo} 

\begin{nt}
\label{schurpair1}
The pair $(A,W)$ is an analogue of the Schur pair, see \cite{Mu} and also \cite{ZO}. 

We have restricted ourself on the case of $1$-quasi elliptic rings in theorem \ref{schurpair} only because of lemma  \ref{lemma7}, \ref{3b} about possibility of normalization. The same is true if we replace words "$1$-quasi elliptic" by "quasi elliptic". The proof is the same.
\end{nt}

We finish this section with the following statement on "purity" of $1$-quasi elliptic subrings of partial differential operators:
\begin{prop}
\label{purity}
Let $B\subset D\subset \hat{D}$ be a $1$-quasi elliptic ring of commuting partial differential operators. Then any ring $B'\subset \hat{D}$ of commuting operators such that $B'\supset B$ is a ring of partial differential operators, i.e. $B'\subset D$.
\end{prop} 

\begin{proof} If $B\subset D$, then by lemma \ref{lemma8}, item \ref{2} the operator $S$ such that $SBS^{-1}=A \subset k[\partial_1]((\partial_2^{-1}))$ belongs to $E$. Since $B'$ is $1$-quasi elliptic, we have also $SB'S^{-1} \subset k[\partial_1]((\partial_2^{-1}))\subset E$. Thus, $B'\subset \hat{D}\cap E= D$. 
\end{proof} 

\subsection{Correspondence between Schur pairs and geometric data}

Now we are going to establish a correspondence between certain $1$-quasi elliptic Schur pairs and geometric data from the generalized Krichever-Parshin correspondence, see \cite{Pa}, \cite{Os}, \cite{Ku} (in fact, we will modify this data, see definition \ref{geomdata} and remark \ref{remark3.5} below). We will consider not all $1$-quasi elliptic Schur pairs, but those which satisfy a condition of strong admissibility  (see definitions below). We emphasize that these pairs include in particular all pairs coming from rings of partial differential operators mentioned in the beginning of previous subsection. As a result, we will obtain a correspondence between $1$-quasi elliptic strongly admissible rings of commuting operators in $\hat{D}$ and geometric data. 

To reach this aim we will need the following "trick lemma". 

\begin{lemma}
\label{trick}
Let $W$ be a closed $k$-subspace $W\subset  k[z_1^{-1}]((z_2))$ 
with $\Sup(W)=\langle z_1^{-i}z_2^{-j}\mbox{\quad }| i,j\ge 0 \rangle$.
Let $\{w_{i,j}, i,j\ge 0\}$ be the unique basis in $W$ with the property $w_{i,j}=z_1^{-i}z_2^{-j}+w_{i,j}^-$, where $w_{i,j}^-\in k[z_1^{-1}][[z_2]]z_2$. Assume that all elements $w_{i,j}$ satisfy the condition $A_{\alpha}$ with $\alpha\ge 1$. 

Then there is an isomorphism  
$$\psi_{\alpha}:W\rightarrow W'$$ 
of $W$ with a closed $k$-subspace $W'\subset k[[u]]((t))$ with $\Sup(W')=\langle u^{i}t^{-j[\alpha ]-i}\mbox{\quad }| i,j\ge 0 \rangle$, where $[\alpha ]$ is the least integer greater or equal to $\alpha$. 
\end{lemma}

\begin{proof} Let's consider the composition of maps $z_1\mapsto u':=z_1^{-1}$, $z_2\mapsto t^{[\alpha ]}$, and $u'\mapsto u=u't$. Due to the conditions of lemma, the images of the elements $w_{i,j}$ will be well defined elements of $k[[u]]((t))$, the composition of these maps is clearly a $k$-linear map which is an isomorphism of $W$ with a closed $k$-subspace $W'\subset k[[u]]((t))$ with described properties. We'll call this composition by $\psi_{\alpha}$.  
\end{proof}

\begin{corol}
\label{trick1}
Let $W$ be a closed $k$-subspace as in lemma and let $\alpha =1$. Then $W'$ in lemma has the property $\Sup(W')=\langle u^{i}t^{-j}\mbox{\quad }| i,j\ge 0, i-j\le 0 \rangle$. 

Moreover, in this case the isomorphism $\psi_1$ induces an isomorphism 
$$
\psi_1:k[z_1^{-1}]((z_2))\cap \Pi_1\rightarrow k[[u]]((t)).
$$
\end{corol}

The proof is clear. 

\begin{nt}
\label{ooo1}
Consider a subspace $W$ in $k[[u]]((t))$ with $\Sup(W)= \langle u^{i}t^{-j}\mbox{\quad }| i,j\ge 0, i-j\le 0 \rangle$ (cf. corollary \ref{trick1}). Let $A$ be a stabilizer subring of $W$: $A\cdot W\subset W$. For any element $a\in A$ we have $\LT(a)\in \Sup(W)$, because for an element $w\in W$ with $\LT (w)=1$ we must have $\LT (aw)=\LT (a)$. So,  $\Sup(A)\subset \Sup(W)$. 
By \cite[lemma 2]{ZO} it is known that the transcendental degree $\trdeg(\Quot(A))\le 2$, where $\Quot (A)$ is the fraction field. 

If we start with a ring $B$ of commuting operators as in theorem \ref{schurpair}(see also remark \ref{schurpair1}) and apply corollary \ref{trick1} to the pair $(W, A)$ from remark \ref{schurpair1}, we'll obtain a pair  $(W,A)$ in $k[[u]]((t))$ as above with  $\trdeg(\Quot(A))=2$ and with another property, which we'll pick out in the following definition.
\end{nt}

\begin{defin}
\label{goodring}
Denote by $\nu_t$ or $\nu_2$ the discrete valuation on the field $k((u))((t))$ with respect to $t$. Denote by $\nu_u$ or $\nu_1$ the discrete valuation on the field $k((u))$. They form a rank two valuation $\nu =\ord_{\Gamma}$ (cf. definition \ref{defin2,5}) on the field $k((u))((t))$: $\nu (a)=(\nu_u(\bar{a}),\nu_t(a))$, where $\bar{a}$ is the residue of the element $at^{-\nu_t(a)}$ in the valuation ring of $\nu_t$. 

For the ring $A\subset k[[u]]((t))$  define 
$$
N_A=GCD\{\nu_t(a), \mbox{\quad} a\in A \mbox{\quad such that \quad} \nu (a)=(0,*)\},
$$ 
where $*$ means any value of the valuation. 

We'll say that the ring $A$ is admissible if there is an element $a\in A$ with $\nu (a)=(1,*)$. 
\end{defin} 

In particular, the ring $A$ obtained from the ring $B$ above is an admissible ring, because $B$ contains an operator of special type (the quasi ellipticity condition). The image of this operator under the transformation from  lemma \ref{trick} satisfies the property from the definition of admissible ring.

Motivated by proposition \ref{chvar}, we'll give also the following definition. 

\begin{defin}
\label{verygood}
For the ring $A\subset k[[u]]((t))$  define 
$$
\tilde{N}_A=GCD\{\nu_t(a), \mbox{\quad} a\in A\}. 
$$ 
We'll say that the ring $A$ is strongly admissible if it is admissible and $\tilde{N}_A=N_A$. 
\end{defin}

\begin{defin}
\label{verygood1}
We say that a $1$-quasi elliptic ring $A\subset k[z_1^{-1}]((z_2))$ from definition \ref{sch} is strongly admissible if its image $\psi_1(A)$ under the transformation from lemma \ref{trick} is strongly admissible.
\end{defin}

\begin{nt}
\label{ooo}
Note that the image $\psi_1(A)$ of a $1$-quasi elliptic ring $A$ is admissible. Conversely, the ring $\psi_1^{-1}(A)$, where $A$ is an admissible ring, is a  $1$-quasi elliptic ring.
\end{nt}

\begin{defin}
\label{ggg1}
For $1$-quasi elliptic commutative ring $B\subset \hat{D}$ one can extend definitions \ref{orter}, \ref{ggg}, and these definitions will be closely related with definitions \ref{goodring}, \ref{verygood}: by theorem \ref{schurpair} $B$ corresponds to a Schur pair $(A,W)$ up to the equivalence, i.e. the ring $A$ is defined up to conjugation by a $1$-admissible operator. Nevertheless, we always have $A\subset \Pi_1$ and $A$ is a $1$-quasi elliptic ring. 

For $1$-quasi elliptic commutative ring $B\subset \hat{D}$ we define numbers $\tilde{N}_B$, $N_B$ to be equal to the numbers $\tilde{N}_A$, $N_A$ (see definition \ref{verygood1}). We say that $B$ is strongly admissible if $A$ is strongly admissible. 
\end{defin}

We claim that our definition is correct, i.e. it don't depend on conjugation of $A$ by a $1$-admissible operator. As we have seen in the proof of corollary \ref{corol3.1} each operator $X$ from $A$ can be written as a finite sum $X=\sum c_{ij}w_{0,0}^{-1}w_{i,j}$, $c_{ij}\in k$. Let $(k,l)$ be a maximal (with respect to the anti lexicographical order) pair of numbers such that $c_{kl}\neq 0$, $k+l\ge i+j$ for all $(i,j)$ with $c_{ij}\neq 0$. It is easy to see that $\nu (\psi_1(X))=(k,l)$. Let $T$ be a $1$-admissible operator. Then using lemma \ref{lemma9} we obtain that $\nu (\psi_1(TXT^{-1}))=\nu (\psi_1(X))=(k,l)$. Thus, the definition of numbers $\tilde{N}_B$, $N_B$ don't depend on conjugation. Again using lemma \ref{lemma9} one can see that this definition coincides with definitions \ref{orter}, \ref{ggg} if $B\subset D$.

Let's recall one more definition (see, for example,~\cite{ZO})
\begin{defin}
\label{techn}
For a $k$-subspace $W$ in $k((u))((t))$, for  $i,j \in \dz\cup \{\infty \}$, $i<j$ 
let
$$
W(i,j) = \frac{W \cap t^ik((u))[[t]]}{W \cap
t^{j}k((u))[[t]]}
$$
be a $k$-subspace in $\frac{t^ik((u))[[t]]}{t^{j}k((u))[[t]]}\simeq k((u))^{j-i}$.
\end{defin}

Note that for spaces $W,A$ as in remark \ref{ooo1} the spaces $W(i,1), A(i,1)$ coincide with the subspaces $W\cap t^ik[[u]][[t]]$, $A\cap t^ik[[u]][[t]]$ of filtration defined by the valuation $\nu_2$. 

\begin{lemma}
\label{qkartier}
Let $A\subset k[[u]][[t]]$ be a commutative $k$-algebra with unity such that $\Sup(A)\subset  \langle u^{i}t^{-j}\mbox{\quad }| i,j\ge 0, i-j\le 0 \rangle$.  
Set $\tilde{A}:=\bigoplus\limits_{n=0}^{\infty}A(-n,1)$.
Assume that $\trdeg(\Quot(A))=2$ and either $\gr (A)=\bigoplus\limits_{n=0}^{\infty}A(-n,1)/A(-n+1,1)$ or $\tilde{A}$ is finitely generated as a $k$-algebra.  
Then 
\begin{enumerate}
\item\label{q1} 
The homogeneous ideal $I=\tilde{A}(-1)$ is prime and it defines a reduced irreducible closed sub-scheme $C$ on the projective surface $X=\Proj \tilde{A}$ which is an ample effective $\dq$-Cartier divisor (i.e. $dC$ is an ample effective Cartier divisor, see remark \ref{divizory}). 
\item\label{q2}
If $A$ is an admissible ring and $N_A=1$, then 
the center $P$ of the valuation $\nu$ induced on the field $\Quot(\tilde{A})$ by the valuation of the two-dimensional local field $k((u))((t))$ is a regular closed point on the curve $C$ as well as on the surface $X$ (cf. \cite[ch.II, ex.4.5]{Ha}).
\end{enumerate}
\end{lemma}

\begin{proof} \ref{q1}) Denote by $i:I \rightarrow \tilde{A}$ the natural embedding. Clearly, we have $I=(i(1))$, where $1\in I_1= \tilde{A}_0$ and $i(1)\in \tilde{A}_1$.
Let $a\in \tilde{A}_k$, $b\in \tilde{A}_l$ be two homogeneous elements such that $a,b\notin I$. This is possible if and only if $\nu_2(a)=-k$, $\nu_2(b)=-l$ (note that such elements exist due to our assumption on the support and transcendental degree of $A$). Therefore $\nu_2(ab)=-k-l$ and the product $ab\in \tilde{A}_{k+l}$ can not belong to $I$, i.e. $I$ is a prime homogeneous ideal. 

By \cite[prop. 2.4.4]{EGAII} the schemes $\Proj \tilde{A}$ and $\Proj \tilde{A}/I$ are integral. So, the ideal $I$ defines a reduced and irreducible closed subscheme $C$ on $X$.  

If $\gr (A)$ is finitely generated, $\tilde{A}$ is also finitely generated over $k$ (it is easy to check that $\tilde{A}$ is generated by elements $\tilde{b_1}, \ldots ,\tilde{b_p}, i(1)$ as $k$-algebra, where $\tilde{b_1}, \ldots ,\tilde{b_p}$ are lifts of generators $b_1,\ldots ,b_p$ of the algebra $\gr (A)$, cf. also \cite[Ch.III, \S 2.9]{Bu}).  By lemma in \cite[ch.III,\S 8]{Mum} there exists $d\in \dn$ such that the graded ring $\tilde{A}^{(d)}=\oplus_{k=0}^{\infty}\tilde{A}_{kd}$ is generated by $\tilde{A}^{(d)}_1$ over $k$ (and $\tilde{A}^{(d)}_1$ is a finitely generated $k$-subspace because of the condition on the support of $A$). 
We claim that $dC$ is a Cartier divisor. Indeed, it is defined by the ideal $I^d=(i(1)^d)$, and $i(1)^d\in \tilde{A}^{(d)}_1$.  By \cite[prop. 2.4.7]{EGAII} we have $\Proj \tilde{A}\simeq \Proj \tilde{A}^{(d)}$ and $\Proj \tilde{A}/I\simeq \Proj \tilde{A}^{(d)}/I^{(d)}$. So, it suffices to show that the ideal $I^{(d)}$ in $\tilde{A}^{(d)}$ defines a Cartier divisor. 
But it is clear, because the open sets $D(x_i)$, where $x_i\in \tilde{A}^{(d)}_1$, form a covering of $X$ and in each set $D(x_i)$ the ideal $I^{(d)}$ is generated by the element $i(1)^d/x_i$. 

At last, $dC$ is a very ample divisor, because it is a hyperplane section in the embedding $\Proj \tilde{A}^{(d)} \hookrightarrow \Proj \tilde{A}^{(d)}_1\simeq \dpp^N$. 

\ref{q2}) Since $X$ is a projective scheme (hence, it is proper over $k$, see e.g. \cite[ch.II, \S 4]{Ha}), there is a unique center $P$ of the valuation $\nu$ by  \cite[ch.II, ex.4.5]{Ha}. Note that $P$ belongs to an affine set $\Spec \tilde{A}_{(x)}$, where $x\in \tilde{A}$ is an element with the properties $\nu (x)=(0,*)$, $x\notin I$ (such an element exists because $N_A=1$), because $\tilde{A}_{(x)}$ belongs to the valuation ring $R_{\nu}$: indeed, if $x\in \tilde{A}_k$, then $\nu_t(x)=k$, and $\nu (a/x^l)=(p,q)$, where $p,q\ge 0$ for any $a\in \tilde{A}_{kl}$. Moreover, it is easy to see that the element $x^{-1}\in k((u))((t))$ (we consider here $\tilde{A}_k=A(-k,1)$ as a vector subspace in $k((u))((t))$, so, $x\in k((u))((t))$) has the property $x^{-1}\in k[[u]][[t]]=k[[u,t]]$. So, we have a natural embedding $\tilde{A}_{(x)}\hookrightarrow k[[u,t]]$. 

Since $A$ is an admissible ring and $N_A=1$, there are elements $u',t'\in \tilde{A}_{(x)}$ with $\nu (u')=(1,0)$ and $\nu (t')=(0,1)$. Denote $B=\tilde{A}_{(x)}$ and let $p\in B$ be the ideal corresponding to $P$. Clearly $u',t'\in p$ and $p=B\cap (u,t)$, where $(u,t)$ is the ideal in $k[[u,t]]$. So, $B/p\simeq k$ and therefore $p$ is a maximal ideal. Since any element $a\in k[[u,t]]$  with $\nu (a)=(0,0)$ is invertible, we have $B_{p}\subset k[[u,t]]$. We'll denote by $p'$ the maximal ideal in $B_p$.  

 We define a linear topology on $B_p$ by taking as open ideals the ideals $M_k:=(u,t)^k\cap B_p$. It is separated, because $\cap (u,t)^k= 0$ in the ring $k[[u,t]]$. Since $p\subset (u,t)$, we have also ${p'}^k\subset M_k$ for all $k$. So, we have the exact sequence of projective systems:
$$
0\rightarrow M_k/{p'}^k \rightarrow B_p/{p'}^k \rightarrow B_p/M_k \rightarrow 0. 
$$  
Note that all natural homomorphisms $M_{k+1}/{p'}^{k+1} \rightarrow M_k/{p'}^k$ are surjective. Indeed, for a given $a\in M_k$ one can find constants $c_i\in k$, $i=0,\ldots k$ such that $a-\sum_{i=0}^{k}c_i{u'}^i{t'}^{k-i}\in M_{k+1}$. Since $\sum_{i=0}^{k}c_i{u'}^i{t'}^{k-i}\in {p'}^k$, it follows that $a$ belongs to the image of the group $M_{k+1}/{p'}^{k+1}$. So, the system $\{M_k/{p'}^k\}$ satisfies the Mittag-Leffler condition and therefore we have the surjective homomorphism of topological rings
$$
\rho :\hat{B_p} \rightarrow \tilde{B_p},
$$
where $\hat{B_p}=\limproj B_p/{p'}^k$, $\tilde{B_p}=\limproj B_p/M_k$. 
Note that $\rho$ preserves the ring $k[u',t']$, and this ring is dense in $\tilde{B_p}$. 

On the other hand, there is a natural homomorphism of topological rings $\rho ': k[[u',t']]\rightarrow \hat{B_p}$ which also preserves the ring $k[u',t']$. So, the composition $\rho \rho '$ is a homomorphism of complete  topological rings that preserves $k[u',t']$, and the ring $k[u',t']$ is dense in both rings. Therefore, it is an isomorphism $k[[u',t']]\simeq \tilde{B_p}$. So, the ring $\tilde{B_p}$ is regular of Krull dimension 2. 

By \cite[corol.11.19]{AM} we have $\dim \hat{B_p}\le 2$, whence $\rho$ must be injective, i.e. it must be an isomorphism. Then by \cite[prop. 11.24]{AM} the ring $B_p$ is a regular ring, i.e. $P$ is a regular closed point on $X$. 

It's easy to see that $(t)\cap B=I_{(x)}$, where $(t)$ is the ideal in the ring $k[[u,t]]$. So, there is an embedding 
$B/I_{(x)}\hookrightarrow k[[u]]$. By analogous arguments as above we have $\widehat{(B/I_{(x)})_p}\simeq k[[u]]$, whence $P$ is a regular point on $C$. 
\end{proof}

\begin{nt}
\label{divizory}
For an arbitrary projective surface $X$ there is a natural homomorphism $Div(X) \rightarrow Z^1(X)$ of the group of Cartier divisors $Div(X)$ to the group of Weil divisors $Z^1(X)$ (in general not injective). The assertion of lemma claims that the scheme defined by the ideal sheaf $\ci^d$ is a locally principal subscheme in $X$ and therefore corresponds to an effective Cartier divisor $D$. Since $X$ is an integral scheme, we have $CaCl(X)\simeq Pic (X)$. By \cite[prop. 6.18, ch.2]{Ha}, $\ci^d\simeq \co (-D)$. The assertion of lemma claims that the sheaf $\co (D)$ is ample (cf. \cite[\S 2.4, appendix]{Ku4}). 
\end{nt}

\begin{lemma}
\label{rcase}
Let $A\subset k[[u]]((t))$ be a strongly admissible ring. 
Then there exists a monic element $t'\in k[[u]]((t))$ with $\nu (t')=(0,N_A)$ and a monic element $u'\in k[[u]]((t))$ with $\nu (u')=(1,0)$ such that $A\subset k[[u']]((t'))\subset k[[u]]((t))$ and in $k[[u']]((t'))$ the ring $A$ has the number $N_A'=1$. 
\end{lemma}

\begin{proof} Since $A$ is strongly admissible, there exist two elements $a,b\in A$ such that $\nu (a)=(0,k_1)$, $\nu (b)=(0,k_2)$ and $GCD(k_1,k_2)=N_A$. Then there exists an invertible monic element $t'\in A_{ab}\subset k[[u]]((t))$ such that $\nu (t')=(0,N_A)$ and therefore there exists a monic element $u'\in A_{ab}$ such that $\nu (u')=(1,0)$. 

Let $v\in A$ be an arbitrary element with $\nu (v)=(k,lN_A)$. Then we can choose a constant $c_{k,l}\in k$ so that $\nu (v-c_{k,l}{u'}^{k}{t'}^{l})=(k_1,l_1N_A)<  (k, lN_A)$. If we continue this procedure, then we have a sequence of constants $c_{k,l}, c_{k_1,l_1}, \ldots$ such that 
$$
v- \sum c_{k_i,l_i}{u'}^{k_i}{t'}^{l_i}=0
$$ 
(it is easy to see that the series in the formula converges). So, $A\subset k[[u']]((t'))$. In the ring $k[[u']]((t'))$ we have $GCD(\nu_{t'} (a), \nu_{t'}(b))=1$. Thus, $N_A'=1$.  
\end{proof}

\begin{prop}
\label{puchok}
Let $W,A\subset k[[u]]((t))$ be subspaces such that $\Sup(W)= \langle u^{i}t^{-j}\mbox{\quad }| i,j\ge 0, i-j\le 0 \rangle$, $A$ is a stabilizer subring of $W$: $A\cdot W\subset W$ (cf. remark \ref{ooo1}). Assume that $\trdeg(\Quot(A))=2$, either $\gr (A)$ or $\tilde{A}$ is a finitely generated $k$-algebra and $A$ is a strongly admissible ring, $A\subset k[[u']]((t'))$ (see lemma \ref{rcase}). 
Set $\tilde{W}:=\bigoplus\limits_{n=0}^{\infty}W(-n,1)$ (see definition \ref{techn}). Then 
\begin{enumerate}
\item\label{puch1}
The sheaf $\cf =\Proj (\tilde{W})$ is a quasi coherent torsion free sheaf\footnote{Here and later in the article we use the non-standard notation $\Proj $ for the quasi-coherent sheaf associated with a graded module.}   on the surface $X$ constructed by $A\subset k[[u']]((t'))$ as in lemma \ref{qkartier}. 
Moreover, we have natural embeddings of ${\co}_P$-modules ${\cf}_P\hookrightarrow k[[u,t]]$ and
of rings $\widehat{\co}_P\hookrightarrow k[[u',t']]\subset k[[u,t]]$, where the last embedding is an isomorphism.
\item\label{puch2}
Let $C'=dC$ be a very ample Cartier divisor on $X$ from lemma \ref{qkartier}. 

The natural embeddings $H^0(X, \cf (nC'))\hookrightarrow \cf (nC')\simeq {\cf}_P\hookrightarrow k[[u,t]]$ coming from the embedding ${\cf}_P\hookrightarrow  k[[u,t]]$ of item \ref{puch1} composed with the homomorphism $k[[u,t]] \rightarrow k[[u,t]]/(u,t)^{ndN_A+1}$ give isomorphisms 
$$
H^0(X, \cf (nC'))\simeq k[[u,t]]/(u,t)^{ndN_A+1}
$$
for each $n\ge 0$. 
\end{enumerate}
\end{prop}

\begin{proof} \ref{puch1}). By the same arguments as in the proof of lemma \ref{qkartier}, item \ref{q2} we have naturally defined embeddings of rings  $\co_P\hookrightarrow k[[u',t']]\subset k[[u,t]]$, $\widehat{\co}_P\simeq k[[u',t']]\hookrightarrow k[[u,t]]$. They define a $\co_P$ and $\widehat{\co}_P$-module structure on $k[[u,t]]$. Since $\tilde{W}$ is a torsion free $\tilde{A}$-module, the sheaf $\cf$ is also torsion free. Thus we have a naturally defined embedding of $\co_P$-modules $\cf_P\hookrightarrow k[[u,t]]$. 

\begin{nt}
Since $W$ contains elements of any valuation $(0,k)$, $k\le 0$ (because of our assumptions on the support of $W$), there are elements $f_1,\ldots ,f_{N_A}\in  \cf_P\subset k[[u,t]]$ such that $\nu (f_i)=(0,i-1)$, $i=1, \ldots N_A$. Clearly, the sheaf $\cf$ can be represented as a direct limit of coherent sheaves, $\cf =\limind \cf_i$ such that $f_1, \ldots , f_{N_A}\in {\cf_i}_P$ for any $i$.  
Consider the map 
\begin{equation}
\label{isom1000}
\co_P^{\oplus N_A}\rightarrow {\cf_i}_P\subset k[[u,t]], \mbox{\quad} (a_1,\ldots a_{N_A})\mapsto a_1f_1+\ldots +a_{N_A}f_{N_A}.
\end{equation}
Clearly, this is an embedding of $\co_P$-modules (since the elements $a_if_i$ have different valuations in the ring $k[[u,t]]$ and there is no torsion, their sum can not be equal to zero). Arguing as in the proof of lemma \ref{qkartier}, item \ref{q2}, we obtain that the map 
$$
\widetilde{\co}_P^{\oplus N_A}\rightarrow \widetilde{\cf_i}_P\simeq k[[u,t]]
$$ 
is an isomorphism of $\widehat{\co}_P$-modules for each $i$ (the completion is with respect to the $M_k$-adic topology). We also have the surjective homomorphism of modules $\rho :\widehat{\cf}_P \rightarrow \widetilde{\cf}_P$. This homomorphism can have a non-trivial kernel, see for examples remark 3.3 and corollary 3.1 in \cite{Ku4}.
\end{nt} 

\ref{puch2}). Since $\cf$ is a torsion free sheaf, we have the canonical embeddings $H^0(X, \cf (nC'))\hookrightarrow \cf_P(nC')$ for all $n\ge 0$. We have $\cf_P(nC')\simeq \cf_P$, and  the  isomorphism of these $\co_P$-modules is given by multiplication by $x^{-1}$, where $x\in \tilde{A}$ is an element with the properties $\nu (x)=(0,-ndN_A)$ as in the proof of item \ref{q2} of lemma \ref{qkartier}. In the proof of item \ref{puch1} we have also seen that ${\cf}_P\hookrightarrow k[[u,t]]$. 

Note that for all $n$ we have $\Proj (\tilde{W}(ndN_A))\simeq \Proj (\tilde{W}^{(dN_A)}(n))$ by \cite[prop. 2.4.7]{EGAII}, and $\Proj (\tilde{W}^{(dN_A)}(n))\simeq \Proj (\tilde{W}^{(dN_A)})(n)\simeq \cf (nC')$ by \cite[ch.II, prop.5.12]{Ha}. Analogously, $\Proj (\tilde{A} (ndN_A))\simeq \co_X(nC')$. 
To prove the rest of the proposition, we need the following lemma. 

\begin{lemma}
\label{kogomologii}
We have $H^0(X,\Proj (\tilde{W}(ndN_A)))= W(-ndN_A,1)$, $H^0(X,\Proj (\tilde{A} (ndN_A)))= A(-ndN_A,1)$ for all $n\ge 0$. 
\end{lemma} 

\begin{proof} The proof is the same for both sheaves. We'll write it for the sheaf $\cf$. 

 By definition, $W(-ndN_A,1)=(\tilde{W}^{(dN_A)}(n))_0\subset H^0(X,\Proj (\tilde{W}(ndN_A)))$. 
Set $\tilde{A}=\bigoplus\limits_{n=0}^{\infty}A'(-n,1)$, where $A'(-n,1)$ are subspaces defined in $k[[u']]((t'))$. Note that $A'(-n,1)=A(-nN_A,1)$, thus $\tilde{W}^{(dN_A)}(n)$ is a graded $\tilde{A}^{(d)}$-module. Recall (see lemma \ref{qkartier}) that the algebra $\tilde{A}^{(d)}$ is generated by $\tilde{A}_d$ as $k$-algebra. 

Let $a\in H^0(X,\Proj (\tilde{W}(ndN_A)))$, $a\notin W(-ndN_A,1)$. Then $a=(a_1,\ldots ,a_k)$, where $a_i\in (\tilde{W}^{(dN_A)}(n))_{(x_i)}$, $x_i\in \tilde{A}_{d}$ are generators of the space $\tilde{A}_{d}$ such that $x_1=1_1^{d}$, and $a_i=a_j$ in $\tilde{A}_{x_ix_j}$ (here we denote by $1_1$ the element $1$ in the component $\tilde{A}_1$).  

We have $a_i=\tilde{a}_i/x_i^{k_i}$ ($\tilde{a}_i\in \tilde{W}^{(dN_A)}(n)_{k_i}=\tilde{W}_{(k_i+n)dN_A}$), $a_1=\tilde{a}_1/x_1^{k_1}$ and $k_1>0$ since $a\notin W(-ndN_A,1)$. Indeed, if $\tilde{a}_1\in (\tilde{W}^{(dN_A)}(n))_{0}=W(-ndN_A,1)$, then $a=\tilde{a}_1$ since $\tilde{W}^{(dN_A)}(n)$ is a torsion free $\tilde{A}^{(d)}$-module, a contradiction. 
So, we have 
$$
\tilde{a}_1\in (\tilde{W}^{(dN_A)}(n))_{k_1}\backslash (\tilde{W}^{(dN_A)}(n))_{k_1-1}
$$ 
(or, equivalently, $(n+k_1)dN_A\ge \nu_t(\tilde{a}_1)>(n+k_1-1)dN_A$). 

Then for $x_i\in \tilde{A}_d\backslash \tilde{A}_{d-1}$ (such an element $x_i$ exists because all elements from $\tilde{A}_{d-1}\subset \tilde{A}_d$ lie in the ideal that defines the divisor $C$) we have $x_i^{k_i}\in \tilde{A}_{dk_i}\backslash \tilde{A}_{dk_i-1}$ (or, equivalently, $\nu_t(x_i^{k_i})=dk_iN_A$) and therefore 
$$
\tilde{a}_1 x_i^{k_i}\in (\tilde{W}^{(dN_A)}(n))_{k_1+k_i}\backslash (\tilde{W}^{(dN_A)}(n))_{k_1+k_i-1},
$$ 
because $\nu_t(\tilde{a}_1 x_i^{k_i})> (n+k_1+k_i-1)dN_A$. 

On the other hand, we have the equality $\tilde{a}_1 x_i^{k_i}= \tilde{a}_i x_1^{k_1}$, and 
$$\tilde{a}_i x_1^{k_1}\in (\tilde{W}^{(dN_A)}(n))_{k_1+k_i-1}\subset (\tilde{W}^{(dN_A)}(n))_{k_1+k_i},
$$ 
because $\nu_t(\tilde{a}_i x_1^{k_1})=\nu_t(\tilde{a}_i)\le (n+k_i+k_1-1)dN_A$, 
a contradiction. So, $a\in W(-ndN_A,1)$.  
\end{proof}

Now we have the embeddings $H^0(X,\cf (nC'))= W(-ndN_A,1) \hookrightarrow \cf (nC')_P\simeq {\cf}_P\hookrightarrow k[[u,t]]$ given by multiplication by $x^{-1}$. Because of our assumptions on the support of $W$, the composition with the homomorphism $k[[u,t]] \rightarrow k[[u,t]]/(u,t)^{ndN_A+1}$ gives isomorphisms 
$$
H^0(X, \cf (nC'))\simeq k[[u,t]]/(u,t)^{ndN_A+1}
$$
for each $n\ge 0$. Note that they don't depend on the choice of the   isomorphism $\cf_P(nC')\simeq \cf_P$. 
\end{proof}

Now we want to establish the correspondence between Schur pairs and geometric data from lemma \ref{qkartier} and  proposition \ref{puchok}. The most convenient way to do this is to establish a categorical equivalence generalizing the equivalence from one-dimensional situation, see \cite[th.4.6]{Mu}, because we have a lot of data involved.

\begin{defin}
\label{geomdata}
We call $(X,C,P,\cf ,\pi , \phi )$ a geometric data of rank $r$ if it consists of the following data:
\begin{enumerate}
\item\label{dat1}
$X$ is a reduced irreducible projective algebraic surface defined over a field $k$;
\item\label{dat2}
$C$ is a reduced irreducible ample $\dq$-Cartier  divisor on $X$;
\item\label{dat3}
$P\in C$ is a closed $k$-point, which is
regular on $C$ and on $X$;
\item\label{dat4}
$$
\pi : \widehat{\co}_{P}\longrightarrow k[[u,t]]
$$
is a ring homomorphism such that
the image of the maximal ideal of the ring $\widehat{\co}_{P}$ lies in the maximal ideal $(u,t)$ of the ring $k[[u,t]]$, and $\nu (\pi (f))=(0,r)$, $\nu (\pi (g))=(1,0)$, where $f\in {\co}_{P}$ is a local equation of the curve $C$ in a neighbourhood of $P$ (since $P$ is a regular point, the ideal sheaf of $C$ at $P$ is generated by one element), and $g\in \co_P$ restricted to $C$ is a local equation of the point $P$ on $C$ (Thus, $g,f$ are generators of the maximal ideal $\cm_P$ in $\co_P$).

Once for all, we choose parameters $u,t$ and fix them (note that $k[[u,t]]$ is a free $\widehat{\co}_P$-module of rank $r$).
\item\label{dat5}
$\cf$ is a torsion free quasi-coherent sheaf on $X$.
\item\label{dat6}
$\phi :{\cf}_P \hookrightarrow k[[u,t]]$ is a ${\co}_P$-module embedding such that the homomorphisms  $$H^0(X, \cf (nC')) \rightarrow k[[u,t]]/(u,t)^{ndr+1}$$ obtained as compositions of natural homomorphisms
$$
H^0(X, \cf (nC'))\hookrightarrow {\cf}(nC')_P\stackrel{f^{nd}}{\simeq}\cf_P\stackrel{\phi}{\hookrightarrow} k[[u,t]] \rightarrow k[[u,t]]/(u,t)^{ndr+1} \mbox{,}
$$
where $C'=dC$ is a very ample divisor,
are isomorphisms for any $n \ge 0$.
\end{enumerate}
Two geometric data $(X,C,P,\cf ,\pi_1 , \phi_1 )$ and $(X,C,P,\cf ,\pi_2 , \phi_2 )$ are identified if the images of the embeddings
(obtained by means of multiplication to $f^{nd}$ as above)
$$
H^0(X, \cf (nC'))\hookrightarrow {\cf}_P\stackrel{\phi_1}{\hookrightarrow} k[[u,t]], \mbox{\quad} H^0(X, \co (nC'))\hookrightarrow \widehat{\co}_P\stackrel{\pi_1}{\hookrightarrow} k[[u,t]]
$$
and
$$
H^0(X, \cf (nC'))\hookrightarrow {\cf}_P\stackrel{\phi_2}{\hookrightarrow} k[[u,t]], \mbox{\quad} H^0(X, \co (nC'))\hookrightarrow \widehat{\co}_P\stackrel{\pi_2}{\hookrightarrow} k[[u,t]]
$$
coincide for any $n \ge 0$.
The set of all quintets of rank $r$ is denoted by $\cq_r$.
\end{defin}

\begin{nt}
\label{remark3.5}
Our definition of a geometric data is slightly more general than analogous definitions in \cite{Pa}, \cite{Os}. In particular, we don't demand that a surface is a Cohen-Macaulay, the divisor $C$ can be not Cartier, but $\dq$-Cartier, and the sheaf $\cf$ can be not locally free. 

These restrictions in definitions of works  \cite{Pa}, \cite{Os} are explained by the fact that geometric data with these restrictions can be reconstructed by subspaces lying in the image of the Krichever-Parshin map described in loc.cit. using certain combinatorial construction. In fact, we don't need this construction in our results. 
\end{nt} 

\begin{nt}
We would like to emphasize that the rank $r$ of the geometric data in general differs from the rank of the sheaf $\cf$, cf. \cite[rem.3.3]{Ku4}.

If $\cf_P$ is a free $\co_P$-module of rank $r$, then $\phi$ induces an isomorphism $\widehat{\cf}_P\simeq k[[u,t]]$ of $\widehat{\co}_P$-modules.
This condition is satisfied if $\cf$ is a coherent sheaf of rank $r$, see \cite[corol.3.1]{Ku4}.
\end{nt}

\begin{defin}
\label{geomcategory}
We define a category $\cq$ of geometric data as follows:
\begin{enumerate}
\item\label{cat1}
The set of objects is defined by
$$
Ob (\cq )=\bigcup_{r\in \sdn} \cq_r.
$$
\item\label{cat2}
A morphism
$$
(\beta , \psi ): (X_1,C_1,P_1,\cf_1 ,\pi_1 , \phi_1 ) \rightarrow (X_2,C_2,P_2,\cf_2 ,\pi_2 , \phi_2 )
$$
of two objects
consists of a morphism $\beta :X_1\rightarrow X_2$ of surfaces and a homomorphism $\psi :\cf_2\rightarrow \beta_*\cf_1$ of sheaves on $X_2$ such that:
\begin{enumerate}
\item\label{cat1.1}
$\beta |_{C_1} :C_1\rightarrow C_2$ is a morphism of curves;
\item\label{cat1.2}
$$
\beta (P_1)=P_2.
$$
\item\label{cat1.3}
There exists a continuous ring isomorphism $h:k[[u,t]] \rightarrow k[[u,t]]$ such that
$$
h(u)=u \mbox{\quad mod \quad} (u^2)+(t), \mbox{\quad} h(t)=t \mbox{\quad mod \quad} (ut)+(t^2),
$$
and the following commutative diagram holds:
$$
\begin{CD}
k[[u,t]] @>h>> k[[u,t]]\\
@AA\pi_2A  @AA\pi_1A\\
\widehat{\co}_{X_2, P_2} @>\beta_{P_1}^{\sharp}>> \widehat{\co}_{X_1, P_1}
\end{CD}
$$
\item\label{cat1.4}
Let's denote by ${\beta}_*(\phi_1)$ a  composition of morphisms of ${\co}_{P_2}$-modules
$${\beta}_*(\phi_1): {\beta_*\cf_1}_{P_2} \rightarrow {\cf_{1}}_{P_1}\hookrightarrow k[[u,t]].$$
There is a $k[[u,t]]$-module isomorphism $\xi :k[[u,t]] \simeq h_*(k[[u,t]])$  such that the following commutative diagram of morphisms of ${\co}_{P_2}$-modules holds:
$$
\begin{CD}
{\cf_2}_{P_2} @>{\psi}>> {\beta_*\cf_1}_{P_2}\\
@VV\phi_2V @VV{\beta}_*(\phi_1)V \\
k[[u,t]] @>\xi >> h_*(k[[u,t]])=k[[u,t]]
\end{CD}
$$
\end{enumerate}
\end{enumerate}
\end{defin}

\begin{defin}
\label{schurdata}
A pair $(A,W)$, where $A,W\subset k[[u]]((t))$, is said to be a Schur pair of rank $r$ if the following conditions are satisfied:
\begin{enumerate}
\item\label{sdat1}
$A$ is a $k$-algebra with unity, 
$\Sup(W)=\langle u^{i}t^{-j}\mbox{\quad }| i,j\ge 0, i-j\le 0 \rangle$ and $A\cdot W\subset W$. 
\item\label{sdat2}
$A$ is a strongly admissible ring (see definition \ref{verygood}), $A$ is finitely generated as $k$-algebra, $\trdeg(\Quot(A))=2$  and $N_A=r$. 
\end{enumerate}
We denote by $\cs_r$ the set of all Schur pairs of rank $r$.
\end{defin}

\begin{nt}
Clearly, for a given Schur pair $(A,W)$ the pair $(\psi_1^{-1}(A), \psi_1^{-1}(W))$ (see corollary \ref{trick1} for definition of $\psi_1$) is a 1-quasi elliptic Schur pair from definition \ref{sch}. Conversely, if $(A,W)$ is a 1-quasi elliptic Schur pair such that $A$ is a strongly admissible ring, then $(\psi_1(A), \psi_1(W))$ is a Schur pair. 
\end{nt} 

\begin{defin}
\label{transfer}
For a given subspace $W\subset k[[u]]((t))$ we define the action of an operator $T\in \Pi_1$ (see corollary \ref{corol1.5}) on $W$ by the formula 
$$
WT=\psi_1(\psi_1^{-1}(W)T).
$$
If $T$ is an $1$-admissible operator (see def. \ref{admissible}) and $A\subset k[[u]]((t))$ is a subring, we define 
$$
T^{-1}AT= \psi_1(T^{-1}\psi_1^{-1}(A)T).
$$ 
\end{defin}

\begin{defin}
\label{schurcategory}
We define the category of Schur pairs $\cs$ as follows:
\begin{enumerate}
\item\label{scat1}
$Ob(\cs )=\bigcup_{r\in \sdn}\cs_r$.
\item\label{scat2}
A morphism $T: (A_2,W_2) \rightarrow (A_1, W_1)$ of two pairs  consists of twisted inclusions
$$
T^{-1}A_2T\hookrightarrow A_1, \mbox{\quad} W_2T \hookrightarrow W_1,
$$
where $T$ is an arbitrary $1$-admissible operator. 
\end{enumerate}
\end{defin} 

In fact, as it follows from definitions, $W_2T=W_1$ as a $k$-subspace in the second inclusion $W_2T \hookrightarrow W_1$ above.

\begin{defin}
\label{mapxi}
Given a geometric data $(X,C,P,\cf ,\pi , \phi )$ of rank $r$ we define a pair of subspaces
$$W,A\subset k[[u]]((t))$$
as follows:

Let $f^d$ be a local generator of the ideal $\co_X(-C')_P$, where $C'=dC$ is a very ample Cartier divisor (cf. definition  \ref{geomdata}, item \ref{dat6}). Then $\nu (\pi (f^d))=(0,r^d)$ in the ring $k[[u,t]]$ and therefore  $\pi (f^d)^{-1}\in k[[u]]((t))$. So, we have natural embeddings for any $n >0$
$$
H^0(X, \cf (nC'))\hookrightarrow {\cf (nC')}_P\simeq f^{-nd} ({\cf }_P) \hookrightarrow k[[u]]((t)) \mbox{,}
$$
where the last embedding is the embedding $f^{-nd}{\cf }_P \stackrel{\phi }{\hookrightarrow } f^{-nd} k[[u,t]] {\hookrightarrow} k[[u]]((t))$ (cf. definition \ref{geomdata}, item~\ref{dat6}). Hence we have the embedding
$$
\chi_1 \; : \; H^0(X\backslash C, \cf )\simeq \limind_{n >0} H^0(X, \cf (nC')) \hookrightarrow k[[u]]((t)) \mbox{.}
$$
We define $W \eqdef \chi_1(H^0(X\backslash C, \cf ))$. Analogously the embedding $H^0(X\backslash C, \co )\hookrightarrow k[[u]]((t))$ is defined (and we'll denote it also by $\chi_1$). We define $A \eqdef \chi_1(H^0(X\backslash C, \co ))$.
\end{defin}

Note that the space $W$ satisfies  condition \ref{sdat1} of definition \ref{schurdata} for the space $W$. 
As it follows from definition, $A\subset k[[u']]((t')) = k[[u]]((t^r))$, where $t'=\pi (f)$, $u'=\pi (g)$ (cf. definition \ref{geomdata}, item \ref{dat4}). Thus, on $A$ there is a filtration $A_n=A'(-n,1)=A(-nr,1)$ induced by the filtration ${t'}^{-n}k[[u']][[t']]$ on the space $k[[u']]((t'))$:
$$
A_n = {A \cap {t'}^{-n}k[[u']][[t']]}=A'(-n,1)= {A \cap {t}^{-nr}k[[u]][[t]]}=A(-nr,1).
$$
Also $\Sup (A)\subset \Sup (W)$, because $1\in \Sup W$ and $W$ is (by construction) a torsion free $A$- module. Clearly, $\trdeg (\Quot (A))=2$ and $A$ is finitely generated as a $k$-algebra. Because of item \ref{dat4} of definition \ref{geomdata} we have $N_A\ge r$, $\tilde{N}_A\ge r$.

\begin{lemma}
\label{claim1}
For a geometric data $(X,C,P,\cf ,\pi , \phi )$ of rank $r$ we have $H^0(X, \co_X(nC'))\simeq A_{nd}$ for all $n\ge 0$, where $C'=dC$ is an ample Cartier divisor. 
\end{lemma}

\begin{proof}
By definition of the ring $A$ we have 
$$
A_{nd}=\{a\in A| f^{nd}a\in k[[u]][[t]]\}=\{a\in A| \nu_t(f^{nd}a)\ge 0\}.
$$
We also have by definition $\chi_1(H^0(X,\co_X(nC')))\subset A_{nd}$. Let $a\in A_{nd}$. Then 
$$a\in \chi_1(H^0(X, \co_X(mC')))$$ 
for some $m\ge n$. Let's show that $a\in \chi_1(H^0(X, \co_X(nC')))$. Assume the converse: $a\notin \chi_1(H^0(X, \co_X(nC')))$.  Below we will identify $a$ with its preimage in $H^0(X\backslash C,\co_X)$ or in $f^{-nd}(\co_{X,P})$. 

There is a neighbourhood $U(P)$ of the point $P$, where the ample Cartier divisor $C'$ is defined by the element $f^d$. Since $a\in A_{nd}$, we have $a\in f^{-nd}(\co_{X,P})$, thus $a|_{U(P)}\in \Gamma (U(P), \co_X(nC'))$. Now we have the following commutative diagram:
$$
\begin{array}{cccc}
&a& \hookrightarrow & H^0(C, \co_X(mC')/\co_X(nC'))\\
&\downarrow & & \downarrow \\
0\rightarrow \Gamma (U(P), \co_X(nC')) \rightarrow & \Gamma (U(P),\co_X(mC')) & \stackrel{\alpha}{\rightarrow} & H^0(U(P)\cap C, \co_X(mC')/\co_X(nC'))
\end{array},
$$
where the vertical arrows are embeddings (the right vertical arrow is an embedding since $\co_X(mC')/\co_X(nC')\simeq \co_X/\co_X((n-m)C'))\otimes_{\co_X}\co_X(mC')$ and $(C, \co_X/\co_X((n-m)C'))$ is an irreducible scheme due to properties of the divisor $C$). 

But $\alpha (a)=0$, a contradiction. Thus, $a\in H^0(X,\co_X(nC'))$.
\end{proof}

\begin{lemma}
\label{claim3}
For a geometric data $(X,C,P,\cf ,\pi , \phi )$ of rank $r$ the corresponding ring $A$ satisfies the following property: there exists a constant $K\ge 0$ such that for all sufficiently big $n\ge 0$ and all $l\le nr-K$ the space $A_n$ contains an element $a$ with $\nu (a)= (-nr, l)$. 

In particular, the ring $A$ is strongly admissible with $N_A=r$. 
\end{lemma}

\begin{proof} As it follows from lemma \ref{claim1}, we have $X\simeq \Proj \bigoplus\limits_{n=0}^{\infty} A_{nd}$ (cf. \cite[lemma 9]{Pa}). Thus, the ring $\tilde{A}^{(d)}=\bigoplus\limits_{n=0}^{\infty} A_{nd}$ is a finitely generated $k$-algebra (cf. \cite[Corol. 10.3]{Za}). 
Then the ring $\tilde{A}=\bigoplus\limits_{n=0}^{\infty} A_{n}$ is finitely generated over $k$,  
since $\tilde{A}=\bigoplus\limits_{l=0}^{d-1}\tilde{A}^{(d,l)}$, where  the modules $\tilde{A}^{(d,l)} = \bigoplus\limits_{i=0}^{\infty} A_{di+l}$, $0<l<d$ are naturally isomorphic to the ideals in $\tilde{A}^{(d)}$, which are finitely generated.

We have 
$$\Proj (\tilde{A}(-1))\simeq \Proj (\tilde{A}^{d,-1}) \mbox{ by \cite[prop.2.4.7]{EGAII}, } \Proj (\tilde{A}^{d,-1}(n))\simeq (\Proj (\tilde{A}^{d,-1}))(nC')$$ 
(see \cite[ch.II,prop.5.12]{Ha}). Thus for all big $n$ $H^0(X, (\Proj (\tilde{A}(-1)))(nC'))\simeq A_{nd-1}$ (cf. \cite[ch.II, ex.5.9]{Ha}; the arguments from the proof of lemma \ref{kogomologii} show that $H^0(X, \Proj (\tilde{A}^{d,-1}(n))= A_{nd-1}$). 
Note that the sheaf $\Proj (\tilde{A}(-1))$ is the ideal sheaf $\ci$ of the divisor $C$ (one can argue as in the proof of lemma \ref{qkartier} and/or note that the localization of the ideal $I=\tilde{A}(-1)$ with respect to any element $a\in A_{n}$ with $\nu_t(a)=-rn$ (so, $a\notin \tilde{A}(-1)$) coincide with the ideal of the valuation $\nu_t$ in the ring $\tilde{A}_{(a)}$). Thus, for all big $n$ we have $H^0(C, \co_C(nC'))\simeq A_{nd}/A_{nd-1}$ and we have the natural embeddings 
\begin{multline}
 H^0(C, \co_C(nC'))\hookrightarrow \co_C(nC')_P,\\
\varphi_n: \co_C(nC')_P\simeq \co_X(nC')_P/\ci (nC')_P \stackrel{f^{nd}}{\hookrightarrow } \co_{X,P}/\ci_P =\\
 =\co_{X,P}/(f)\simeq \co_{C,P} \hookrightarrow k[[u,t]]/(t)\simeq k[[u]]
\end{multline}
such that the image of $H^0(C, \co_C(nC'))$ in $k[[u,t]]/(t)$ coincide with the image of the map \\
$A_{nd}/A_{nd-1}\stackrel{f^{nd}}{\hookrightarrow } k[[u,t]]/(t)$.  

On the other hand, for the sheaf $\cf_n=\co_C(nC')$ we have analogous construction of a subspace $W_n$ in $k((u))$ coming from one-dimensional Krichever correspondence (cf. \cite{Pa}). Namely, for each $q\ge 0$ we have natural embeddings 
$$
H^0(C, \cf_n(qP))\hookrightarrow \cf_n(qP)_P\simeq g^{-q}(\cf_{n,P})\hookrightarrow k((u)), 
$$
where the last embedding is the embedding 
$$g^{-q}\cf_{n,P}\stackrel{\varphi_n}{\hookrightarrow} g^{-q}k[[u]]=u^{-q}k[[u]] \hookrightarrow k((u))$$ 
(cf. definition \ref{geomdata}, item \ref{dat4}; we identify here the element $g$ from definition and its image in $k[[u]]$). Hence we have the embedding (cf. definition \ref{mapxi}) 
$
H^0(C\backslash P, \cf_n)\hookrightarrow k((u)),
$ 
whose image we denote by $W_n$. 
If $d'P$ is a very ample Cartier divisor, then arguing as in lemma \ref{claim1} we get $H^0(C, \cf_n(qd'P))\simeq W_{n,qd'}$, where $W_{n, qd'}= W_n \cap u^{-qd'}k[[u]]$. For big $n$ by the Riemann-Roch theorem for curves we get $\dim_k(H^0(C, \cf_n(qd'P)))-\dim_k(H^0(C, \cf_n((q-1)d'P)))= d'$ for all $q\ge 0$. Thus, $\dim_k (W_{n,qd'}/W_{n,(q-1)d'})=d'$ and therefore the space $W_n$ contain an element with any given negative value of the valuation $\nu_u$. 

Now consider the sheaf ${\cf'}_n=\cf_n(-d'P)$. Then for each $q\ge 0$ we have natural embeddings 
$$
H^0(C, {\cf'}_n(qP))\hookrightarrow {\cf'}_n(qP)_P\simeq g^{-q}({\cf'}_{n,P})\hookrightarrow k((u)), 
$$
where the last embedding is the embedding $g^{-q}{\cf'}_{n,P}\simeq g^{-q+d'}\cf_{n,P}\stackrel{g^{-d'}\varphi_n}{\hookrightarrow} u^{-q}k[[u]] \hookrightarrow k((u))$. Hence we have the embedding  
$
H^0(C\backslash P, {\cf'}_n)\hookrightarrow k((u)),
$ 
whose image ${W'}_n=g^{-d'}W_n$. Again by the Riemann-Roch theorem we obtain that for sufficiently big $n$ the space ${W'}_n$ contains elements of any given negative value of the valuation $\nu_u$. Moreover, it follows that there exists a constant $K\ge 0$ such that for all sufficiently big $n$ the space $W_n$ contains elements of any given value $l$ of the valuation $\nu_u$ if $l\le ndr -K$ (because by definition \ref{geomdata}, item \ref{dat6} the space $W_n$ contains no elements with valuation greater than $ndr$). In particular, it follows that the space $A_{nd}$ contains elements of any given value $(-ndr,l)$ of the valuation $\nu$ if $l\le ndr -K$. Thus, the ring $A$ is admissible. 

Now we can repeat all arguments above for the sheaf $\ci (nC')|_C$. Note that $H^0(C,  \ci (nC')|_C)\simeq A_{nd-1}/A_{nd-2}$, and the image of the embedding $H^0(C, \ci (nC')|_C)\hookrightarrow k[[u,t]]/(t)$ is $f^{nd-1}(A_{nd-1}) \mod (t)$. Therefore, for sufficiently big $n$ the space $A_{nd-1}$ contains elements of any given value $(-(nd-1)r,l)$ of the valuation $\nu$ if $l\le (nd-1)r -K$. Thus, $N_A=r$ and  the ring $A$ is strongly admissible, because $\tilde{N}_A | N_A$ and $\tilde{N}_A\ge r$. 

Continuing this line of reasoning, one can obtain that for sufficiently big $n$ each space $A_n$ contains elements of any given value $(-nr,l)$ of the valuation $\nu$ if $l\le nr -K$. 
\end{proof}

\begin{lemma}
\label{claim2}
Let $(A,W)$ be a Schur pair of rank $r$. Then $\tilde{A}=\bigoplus\limits_{n=0}^{\infty}A_n$ and $\gr (A)=\bigoplus\limits_{n=0}^{\infty}A_n/A_{n-1}$ are finitely generated $k$-algebras (cf. lemma \ref{qkartier}).
\end{lemma}

\begin{proof} Let $A$ be generated by the elements $t_1,\ldots ,t_m$ as a $k$-algebra. Denote by $t_{1,s_1}, \ldots ,t_{m,s_m}$ the corresponding homogeneous elements in $\tilde{A}$, where for each $i$ $s_i$ means the minimal number such that $t_i\in A(-s_i, 1)$. Without loss of generality we can assume that generators contain elements $a,b$ with $GCD (\nu_t(a), \nu_t(b))=r$, $\nu (a)=(0,\nu_t(a))$, $\nu (b)=(0,\nu_t(b))$, and an element $c$ with $\nu (c)=(1,*)$ (because $A$ is a strongly admissible ring).

Consider the finitely generated $k$-subalgebra $\tilde{A}_1=k[1_1, t_{1,s_1},\ldots ,t_{m,s_m}] \subset \tilde{A}$ (here we denote by $1_1$ the element $1\in A(-1,1)$).  Arguing as in the proof of lemma \ref{qkartier} and proposition \ref{puchok}, we can construct a geometric data $(X,C,P,\cf ,\pi , \phi )$ of rank $r$ from definition \ref{geomdata}. Note that $H^0(X\backslash C, \co_X)\simeq (\tilde{A}_1)_{(1_1)}\simeq A$. Thus, the space constructed by the data in definition \ref{mapxi} will coincide with $A$. Then by lemma \ref{claim1} $H^0(X, \co_X(nC'))\simeq A_{nd}$, where $C'=dC$ is an ample Cartier divisor. Therefore, the ring $\tilde{A}^{(d)}$ is a finitely generated $k$-algebra (see e.g. \cite[corol. 10.3]{Za}). Hence $\tilde{A}$ is a finitely generated $k$-algebra (cf. the beginning of the proof of lemma \ref{claim3}). The algebra $\gr (A)$ is finitely generated because $\gr (A)\simeq \tilde{A}/(1_1)$. 
\end{proof}

\begin{defin}
\label{functor}
We define a map $\chi :Ob(\cq )\rightarrow Ob(\cs )$ as follows.

If $q=(X,C,P,\cf ,\pi , \phi ) \in Ob (\cq )$ is an element of $\cq_r$, then we define  
$$\chi (q)=(\chi_1 (H^0(X\backslash C, \co_X)), \chi_1 (H^0(X\backslash C, \cf )))\in \cs_r.$$
As it follows from remarks above and lemma \ref{claim3}, $\chi (q)$ is a Schur pair of rank $r$. 
\end{defin}

The following lemma will be needed to prove equivalence of the categories $\cq$ and $\cs$. 

\begin{lemma}
\label{vspomog}
Let $u',v'\in k[[u,t]]$ be monic elements with $\nu (u')=(1,0)$, $\nu (v')=(0,1)$. Then there exists an admissible operator $T\in \Adm_{\alpha }$ such that $T^{-1}uT=u'$, $T^{-1}vT=v'$. 
\end{lemma} 

This is an easy consequence of lemma \ref{lemma8}, \ref{4} and lemma \ref{lemma7}, \ref{3b}. 

Recall that for a given category $\Upsilon$ by $\Upsilon^{op}$ we denote the category with the same objects but with inverse arrows. 

\begin{theo}
\label{dannye}
The map $\chi$ from definition \ref{functor} induces a contravariant functor 
$$
\chi :\cq \rightarrow \cs^{op} 
$$    
which makes these categories equivalent. 
\end{theo} 

\begin{proof} First let's show that the map $\chi$ induces a bijection $\chi_r: \cq_r \rightarrow \cs_r$. 

It will follow from lemma \ref{claim2}, lemma \ref{claim1}, proposition \ref{puchok}, lemma \ref{qkartier}, lemma \ref{kogomologii} and the following statement (cf. e.g. \cite[lemma 9]{Pa}). Suppose that $X$ is a projective scheme over a field, $\cf$ is a coherent sheaf on $X$, and $C'$ is an ample Cartier divisor on $X$. Then $X\simeq \Proj (S)$ and $\cf \simeq\Proj (F)$, where $S=\oplus_{m\ge 0}H^0(X, \co_X(mC'))$, $F=\oplus_{m\ge 0} H^0(X, \cf (mC'))$. 

Having this statement in mind, starting with geometric data $q=(X,C,P,\cf ,\pi , \phi )$ of rank $r$, we can reconstruct it from the Schur pair $\chi (q)= (A,W)$ of rank $r$ as follows. $X\simeq \Proj (\bigoplus\limits_{n=0}^{\infty} A_{nd})$ (see lemma \ref{claim1}), and $\Proj (\bigoplus\limits_{n=0}^{\infty} A_{nd})\simeq \Proj \tilde{A}$. The divisor $C$ and the point $P$ are uniquely reconstructible by the discrete valuation $\nu_t$ and the valuation $\nu$ on the ring $k[[u]]((t))$. By \cite[prop.2.6.5]{EGAII} the composition of canonical homomorphisms $\Gamma_*(\cf )\rightarrow \Gamma_*(\Proj(\Gamma_*(\cf ))) \rightarrow \Gamma_*(\cf )$ (for notation see loc. cit.) is the identity isomorphism. In particular, the homomorphism $\Gamma_*(\Proj(\Gamma_*(\cf ))) \rightarrow \Gamma_*(\cf )$ is surjective. By definition of geometric data $\Proj (\Gamma_*(\cf ))\simeq \Proj (\bigoplus\limits_{n=0}^{\infty}W(-ndr, 1))$  (and $\Proj (\bigoplus\limits_{n=0}^{\infty}W(-ndr, 1)) \simeq \Proj \tilde{W}$ by \cite[prop. 2.4.7]{EGAII}). By lemma \ref{kogomologii} $\Gamma_*(\Proj(\Gamma_*(\cf )))= \Gamma_*(\cf )$. Therefore, the canonical homomorphism $\Proj(\Gamma_*(\cf )) \rightarrow \cf$ must be an isomorphism (otherwise there exists $n\gg 0$ such that $H^0(X, \Proj(\Gamma_*(\cf (nC'))) \rightarrow H^0(X, \cf (nC'))$ is not an isomorphism). So, $\cf \simeq \Proj (\tilde{W})$. The homomorphisms $\pi$ and $\phi$ are naturally defined by the embedding of the subspaces $A,W$ in $k[[u]]((t))$.  
 
 Conversely, starting from a pair $(A,W)\in \cs_r$, by lemma \ref{claim2}, lemma \ref{qkartier}, proposition \ref{puchok} we can construct a geometric data $q\in \cq_r$. Applying to it the map $\chi$, we obtain the same pair (cf. the proof of lemma \ref{claim2}). 

Now let's show how to define the functor $\chi$ on the morphisms. Let's start with a morphism $(\beta, \psi ):q_1 \rightarrow q_2$ between two data. We have an automorphism $h: k[[u,t]]\rightarrow k[[u,t]]$ of definition \ref{geomcategory}, \ref{cat1.3}. Because of lemma \ref{vspomog}, there is an admissible operator $T_1\in \Adm_{1}$ such that 
$$
T_1^{-1}uT_1=h(u), \mbox{\quad} T_1^{-1}vT_1=h(v).
$$
Moreover, as follows from the proof of lemma \ref{lemma8}, we can find $T_1$ such that $1\cdot T_1=1$. 

The ring automorphism $h$ extends to a ring automorphism $h: k[[u]]((t))\rightarrow k[[u]]((t))$ in an obvious way. Thus 
$$
k[[u]]((t))\ni f(u,v) \mapsto f(h(u),h(v))=f(T_1^{-1}uT_1, T_1^{-1}vT_1)=T_1^{-1}f(u,v)T_1\in  k[[u]]((t)).
$$  
The $k[[u,t]]$-module isomorphism $\xi: k[[u,t]]\rightarrow h_*k[[u,t]]$ of definition \ref{geomcategory}, \ref{cat1.4} is given by the multiplication of a single invertible element $\xi\in k[[u,t]]^*$. It determines a $1$-admissible operator $T_2=\psi_1^{-1}(\xi )$ (see corollary \ref{trick1}). Since it is an operator having only constant coefficients, $T_2^{-1}AT_2=A$ for every subset $A\subset  k[[u]]((t))$. 

Now let $(A_i, W_i)=\chi (q_i)$, $i=1,2$.  Since we have from definitions \ref{mapxi} and \ref{geomcategory}, \ref{cat1.3} that 
$$
\begin{CD}
H^0(X_2\backslash C_2, \co_2) @>\beta^*>> H^0(X_1\backslash C_1, \co_1) \\
@VV\chi_2V @VV\chi_1V \\
k[[u]]((t)) @>h>> k[[u]]((t)), 
\end{CD}
$$
we obtain 
$$
T_1^{-1}T_2^{-1}A_2T_2T_1=T_1^{-1}A_2T_1=h(A_2)=h\chi_2(H^0(X_2\backslash C_2, \co_2))\subset \chi_1(H^0(X_1\backslash C_1, \co_1))=A_1.
$$
On the other hand, we have from definitions \ref{mapxi} and \ref{geomcategory}, \ref{cat1.4} that 
$$
\begin{CD}
H^0(X_2\backslash C_2, \cf_2) @>\widehat{\psi}>>H^0(X_2\backslash C_2, {\beta}_*\cf_1) = H^0(X_1\backslash C_1, \cf_1)\\
@VV\chi_2V @VV\chi_1V \\
k[[u]]((t)) @>\xi >> h_*(k[[u]]((t)))=k[[u]]((t)).
\end{CD}
$$
The isomorphism $\xi$ is completely determined by its image $\xi (1)=1\cdot T_2$. 
Every element of the $k[[u]]((t))$-module $k[[u]]((t))$ is of the form $a\cdot 1$, where $a\in k[[u]]((t))$. Hence
$$
\xi (a\cdot 1)=h(a)\cdot \xi (1)= \xi (1) T_1^{-1}aT_1.
$$
Therefore, we conclude that $\xi =T\eqdef T_1T_2$, because of the following consistency:
$$
\xi (a\cdot 1)=  1\cdot T_2\cdot T_1^{-1}aT_1=1\cdot T\cdot T^{-1}aT =aT.  
$$
Thus we have 
$$
W_2 T= \xi (\chi_2(H^0(X_2\backslash C_2, \cf_2)))\subset \chi_1(H^0(X_1\backslash C_1, \cf_1))=W_1.
$$
$T$ is a $1$-admissible operator and we have $T^{-1}A_2T\subset A_1$ and $W_2T\subset W_1$. 
Hence we have constructed a morphism 
$$\chi (\beta ,\psi ): (A_2,W_2) \rightarrow (A_1,W_1)$$
and our functor is defined. 

Let's show that $\chi$ gives an anti equivalence of categories. It is remain to construct an inverse functor on morphisms in $\cs$. 

Let $T: (A_2,W_2) \rightarrow (A_1,W_1)$ be a morphism between Schur pairs defined by an admissible operator $T\in \Adm_{1}$. It means that we have 
\begin{equation}
\label{incl}
T^{-1}A_2T\subset A_1 \mbox{\quad and \quad} W_2T\subset W_1.
\end{equation}
Let $X_i$ be the projective surface defined by $A_i$ and $\cf_i$ be the torsion free sheaf corresponding to $W_i$, $i=1,2$. Note that $W_1$ has a natural $T^{-1}A_2T$-module structure. Thus the inclusions (\ref{incl}) define a morphism (since conjugation and multiplication by $T$ preserves the filtration on $A_2$ and on $W_2$ and therefore an inclusion of graded rings and modules is defined) $\beta :X_1 \rightarrow X_2$ and a sheaf homomorphism $\psi :\cf_2 \rightarrow \beta_*\cf_1$. As it follows from the inclusion of graded rings, the properties \ref{cat1.1} and \ref{cat1.2} of definition \ref{geomcategory} for the morphism $\beta$ hold. 

Since $T$ is $1$-admissible, we have $T^{-1}k[[u,t]]T\simeq k[[u,t]]$, which gives an isomorphism $h: k[[u,t]] \rightarrow k[[u,t]]$. Moreover, $T$ gives an isomorphism between $k[[u]]((t))$-module $k[[u]]((t))$ and $T^{-1}k[[u]]((t))T$-module $k[[u]]((t))T$. Since $k[[u]]((t))$ is generated by the identity element $1$ as a $k[[u]]((t))$-module, $T:k[[u]]((t))\rightarrow k[[u]]((t))$ is determined by its image $\xi \eqdef 1\cdot T\in k[[u,t]]$. Then $\xi$ is an invertible element, $\xi\in k[[u,t]]^*$. Every element of $k[[u]]((t))$ is uniquely expressed as $a\cdot 1$, where $a\in k[[u]]((t))$. We have 
$$
T(a\cdot 1)=(1\cdot T)T^{-1}aT= h(a)\xi .
$$
It is easy to check that $h$ satisfies \ref{cat1.3} of definition \ref{geomcategory} and $\xi$ defines a $k[[u,t]]$-module isomorphism 
$$
\xi :k[[u,t]]\rightarrow k[[u,t]]
$$
which satisfies \ref{cat1.4} of definition \ref{geomcategory}. This completes the proof.  
\end{proof}

Denote the set of isomorphism classes of Schur pairs by $\cs /\Adm_{1}$ and denote the set of isomorphism classes of geometric data by $\cm$. By theorem \ref{dannye}, we obtain

\begin{corol}
\label{dannye1}
There is a natural bijection 
$$
\Phi : \cm \rightarrow \cs /{\Adm}_{1}.
$$   
\end{corol}

Combining theorem \ref{schurpair} and theorem \ref{dannye}, we obtain

\begin{theo}
\label{dannye2}
There is a one to one correspondence between the set of classes of equivalent $1$-quasi elliptic strongly admissible finitely generated rings of operators in $\hat{D}$ (see definitions \ref{elliptic}, \ref{rings}, \ref{ggg1}) and the set of isomorphism classes of geometric data $\cm$ (see definitions \ref{geomdata}, \ref{geomcategory}). 
\end{theo}

\begin{nt}
\label{zakl1}
A natural question arises: are a category of commutative algebras of operators and the category of Schur pairs equivalent? 

The answer is negative already in one-dimensional case, see \cite{Mu}, introduction. It is possible to define a category of commutative algebras of  operators in a natural way. But it does not become equivalent with the category of Schur pairs and the category of geometric data we have defined, since in the construction of a Schur pair by a ring of operators in theorem \ref{schurpair} we need to choose operators $L_1,L_2$, and by choosing other operators, we come to another Schur pair, which is isomorphic to the first one. 
\end{nt}

\begin{nt}
\label{zakl2}
It should be possible to extend the category of geometric data to include also schemes of non-finite type over $k$, and prove the equivalence of this category with an extended category of Schur pairs with the ring $A$ not finitely generated over $k$.
\end{nt}

\begin{nt}
\label{zakl3}
It would be interesting to find geometric conditions describing those geometric data that correspond to $1$-quasi-elliptic rings in the ring $D\subset \hat{D}$. See works \cite{ZhM}, \cite{Ku4}, where several results in this direction are obtained. 
\end{nt}

\begin{nt}
\label{BA}
One can also introduce for the ring $\hat{D}$ and for a surface from definition \ref{geomdata} a natural generalization of the notion of formal Baker-Akhieser module (cf. \cite[Introduction]{ZhM}) or of formal Baker-Akhieser functions as eigenvectors of a ring $B$ from theorem \ref{dannye2} (cf. \cite[\S 4]{Kr}), though it will be in general different from those considered in \cite{Kr} or \cite{ZhM}.  

Namely, consider the expression $e^{\varepsilon}=\exp (x_1z_1^{-1}+x_2z_2^{-1})$ and define the action 
$$
\partial_1(e^{\varepsilon})=z_1^{-1}e^{\varepsilon}, \mbox{\quad} \partial_2(e^{\varepsilon})=z_2^{-1}e^{\varepsilon},
$$
$$
\partial_1^{-1}(e^{\varepsilon})=z_1e^{\varepsilon}, \mbox{\quad} \partial_2^{-1}(e^{\varepsilon})=z_2e^{\varepsilon}. 
$$
Now let's define the $\hat{D}$-module $M=\hat{D}e^{\varepsilon}$. Let's call its elements as formal Baker-Akhieser (BA) functions. 

Let $B$, $P,Q,L_1,L_2$, $S$ be the ring and operators considered in section \ref{reduction}. Let's define the formal BA-function corresponding to $B$ as 
$$
\psi_B(x,z)=S^{-1}(e^{\varepsilon}).
$$
Then we have
$$
P\psi_B(x,z)=z_2^{-k}\psi_B(x,z), \mbox{\quad} Q\psi_B(x,z)=z_1^{-1}z_2^{1-l}\psi_B(x,z).
$$
Note that the eigenvalues are different from the symbols of operators even if $P,Q$ are partial differential operators as in \cite[\S 4]{Kr}. 

In general, for arbitrary element $b\in B$ we have $b\psi_B(x,z)=a\psi_B(x,z)$, where $a$ is a series in $z_1,z_2$. If we apply the change of variables $\psi_1$ from corollary \ref{trick1} to the element $a$, we obtain a series in $u,t$ which is a representation of the meromorphic function on the surface $X$ corresponding to the element $b$ in terms of local parameters of the point $P$ (see definition \ref{geomdata}). Thus, $M$ can be thought of as an analogue of the BA-module, and $\psi_1(\psi_B(x,z))$ can be thought of as an analogue of the BA-function from \cite[\S 4]{Kr}. 
\end{nt}

\section{Examples}
\label{examples}

As an advertisement of our constructions let's give  several examples of commuting operators in the ring $\hat{D}$ (for more details on calculations see \cite{Ku3}).

\begin{ex}
\label{ex1}
In one dimensional situation, using the Sato theorem, one can obtain old known example of Burchnall and Chaundy of commuting ordinary differential operators corresponding to a cuspidal curve, if we take $W=\langle 1+t, t^{-i}, i\ge 1\rangle$, $A=k[t^{-2}, t^{-3}]$:
$$
P= \partial_x^2 - 2(1-x)^{-2}, \mbox{\quad} Q=\partial_x^3-3(1-x)^{-2}\partial_x-3(1-x)^{-3}.
$$
\end{ex}

\begin{ex}
\label{advertisement}
Let's take a subspace $W=\langle 1+t, t^{-i}u^j, i\ge 1, 0\le j\le i \rangle \subset k[[u]]((t))$. One can easily check that its ring of stabilizers contains elements $t^{-2}, t^{-3}, ut^{-2}$. So, it is strongly admissible. The maximal ring of stabilizers will be infinitely generated over $k$. 
The Schur pair $(W,A)$ with a finitely generated ring $A$ containing the elements above corresponds to a geometric data with a surface being singular toric surface. 

The operators corresponding to the elements $t^{-2}, ut^{-2}$ in the ring of commuting  operators corresponding to $A$ (the operators satisfying the definition of quasi ellipticity, cf. also corollary \ref{corol3.1}) are
$$
P=\partial_2^2-2\frac{1}{(1-x_2)^2}(:\exp (-x_1\partial_1):),
$$
$$
Q=\partial_1\partial_2+\frac{1}{1-x_2}(:\exp (-x_1\partial_1):)\partial_1,
$$  
where $(:\exp (-x_1\partial_1):)=1-x_1\partial_1+x_1^2\partial_1^2/2!-x_1^3\partial_1^3/3!+\ldots$.  The operator corresponding to the element $t^{-3}$ is 
$$
P'=\partial_2^3-3\frac{1}{(1-x_2)^{2}}(:\exp (-x_1\partial_1):)\partial_2-3\frac{1}{(1-x_2)^{3}}(:\exp (-x_1\partial_1):).
$$
Thus, these operators  are very similar to the operators from previous example. This similarity goes further: if we derive equations of isospectral deformations of the operators above (cf. \cite[\S 4]{Mul} and \cite[\S 6]{Zhe}), we obtain the following equations of the corresponding Sato-Wilson system (cf. \cite[\S 4]{Zhe}):
\begin{equation}
\label{isospectral}
\frac{\partial s_1}{\partial t_1}= \frac{1}{4}(s_1)_{x_2x_2x_2}-\frac{3}{2}(s_1)_{x_2}^2, \mbox{\quad} \frac{\partial s_1}{\partial t_2}=-(s_1)_{x_2}(s_1)_{x_1}-\frac{1}{2}(s_1)_{x_2x_2}\partial_1, \mbox{\quad} 
\end{equation}
$$
\frac{\partial s_1}{\partial t_3}= -(s_1)_{x_1}^2-(s_1)_{x_1x_2}\partial_1 - (s_1)_{x_2}\partial_1^2,
$$  
where $s_1(t_1,t_2,t_3)=s_1(t)$ is the first coefficient of the operator $S(t)=1+s_1(t)\partial_2^{-1}+\ldots$, and $S(0)=S$ is the conjugating operator:  $W=W_0S$, $P=S\partial_2^2S^{-1}$.  Notably $s_1(0)= \frac{1}{1-x_2}(:\exp (-x_1\partial_1):)$ is a solution of the equations above. This corresponds to the following fact from one-dimensional KP theory: the function $u(x)=(x^{-1})_x$ is the rational solution of the KdV equation (and this function is the halved coefficient of the operator $P$ in example \ref{ex1}). 
\end{ex} 

\begin{nt}
\label{posl}
A simple analysis of equations (\ref{isospectral}) show that even if we start with a commutative ring of partial differential operators (what means that $s_1(0)\in k[[x_1,x_2]][\partial_1]=D_1$), the isospectral deformations will not be partial differential operators, but operators in $\hat{D}$, since $s_1(t)\notin D_1$ for general $t$. Thus, the ring $\hat{D}$ appears quite natural. This situation is similar to the problem of describing commutative rings of ordinary differential operators with polynomial coefficients (cf. \cite{Mir}, \cite{Mok} for explicit examples of such rings) in dimension one. In one dimensional KP theory, if we start with a commutative ring of ordinary differential operators with polynomial coefficients, its isospectral  deformations (which are connected with solutions of the KP equation) will consist of operators with not polynomial coefficients though they will still be ordinary differential operators. 
\end{nt}

\begin{ex}
\label{ex3}
In this example we show how already known examples of commuting partial differential operators corresponding to quantum Calogero-Moser system and rings of quasi-invariants (see \cite{Ch}) fit into our classification. 

Recall that the rings in these examples consist of operators commuting with Schr\''odinger operator $L=\partial_1^2+\partial_2^2-u(x_1,x_2)$, where $u$ is a function of special type given by explicit formulae in three cases: rational, trigonometric and elliptic. In all cases the rings of highest symbols of commuting operators are described (they are called as rings of quasi-invariants, see \cite{Ch}). Thus, the rings of quasi-invariants are $k$-subalgebras in the ring of polynomials (in two variables in our case). As it follows from definition and description of these rings in \cite{Ch}, the corresponding rings of commuting partial  differential operators satisfy assumptions of proposition \ref{chvar} and lemma \ref{lemma5}. Thus, after a linear change of variables they become a $1$-quasi elliptic strongly admissible rings (by proposition \ref{chvar}) and therefore correspond to $1$-quasi-elliptic Schur pairs. If the ring of quasi-invariants is finitely generated as a $k$-algebra (cf. proposition \ref{techn5.2}), then the ring of commuting differential operators corresponds to a Schur pair from definition \ref{schurdata} (by applying the map $\psi_1$ from corollary \ref{trick1} to the corresponding $1$-quasi elliptic Schur pair from theorem \ref{schurpair}) and therefore it also corresponds to a geometric data from definition \ref{geomdata} by theorem \ref{dannye}. 

For example, the operators 
$$
L_1=\partial_1+\partial_2, \mbox{\quad} L_2=\partial_1^2+\partial_2^2-m(m+1)\wp (x_1-x_2)
$$
that define a quantum Calogero-Moser system (here $\wp (z)$ is the Weierstrass function of a smooth elliptic curve), after applying the $k$-linear change of variables ${\partial'}_2=\partial_1+\partial_2$, ${\partial'}_1=\partial_1$, $x_2'=x_2$, $x_1'=x_1-x_2-c$, $c\in \dc$ become equal to 
$$
L_1={\partial'}_2, \mbox{\quad} L_2=2{\partial'}_1^2-2{\partial'}_1{\partial'}_2+{\partial'}_2^2-m(m+1)\wp (c+x_1').
$$ 
We choose a constant $c$ here in such a way that the Taylor series of the function $\wp (z)-z^{-2}$ in a neighbourhood of zero and all its derivatives converge at $z=c$. In this case we can represent $\wp (c+x_1')$ as a formal Taylor series belonging to $\dc [[x_1']]$. Note that any ring of commuting operators containing these operators contains also the operator $L_2'=L_2-L_1^2$ and $\ord_{\Gamma}(L_2')=(1,1)$, $\ord_{\Gamma}(L_1)=(0,1)$. 
Note that both operators $L_1,L_2'$ satisfy the condition $A_1$. 
Therefore, any ring $B$ of commuting operators containing these operators is $1$-quasi elliptic strongly admissible with $N_B=1$. We would like to emphasize that the projective surface $X$ in the geometric data corresponding to this commutative ring of partial differential operators is naturally isomorphic to the projectivization of the affine spectral variety defined by this ring (cf. \cite[rem.5.3]{BEGa}) offered by Krichever in \cite{Kr}. For further geometric properties of the surface $X$ as well as of the geometric data (corresponding to any commutative rings of partial differential operators or operators in $\hat{D}$) we refer to recent works \cite{ZhM}, \cite{Ku4}. 
\end{ex}

At the end we would like to prove one statement about geometric properties of the surface $X$ corresponding to a maximal commutative subring of partial differential operators. This statement recovers a number of results in works  \cite{FV}, \cite{EG}, \cite{BEG}, \cite{FV2} (cf. \cite[rem. 3.17]{Ch})   
claiming that the affine spectral varieties of commutative rings of partial differential operators corresponding to certain rings of quasi invariants  are Cohen-Macaulay. 

To formulate this statement recall one construction (without details) given in section 3.2 of \cite{Ku4}. For a given integral
two-dimensional scheme $X$ of finite type over a field $k$ (or over
the integers) there is a "minimal" Cohen-Macaulay scheme $CM(X)$ and
a finite morphism $CM(X) \rightarrow X$ (and a finite morphism from
the normalization of $X$ to $CM(X)$). The construction generalizes
the known construction of normalisation of a scheme. For the ring $A$ we denote by $CM(A)$ its Cohen-Macaulaysation. 
\begin{theo}
\label{last}
Let $(A,W)$ be a Schur pair of rank $r$ such that $W$ is a finitely generated $A$-module. Then $(CM(A),W)$ is also a Schur pair of rank $r$.

In particular, if $(A,W)$ corresponds to a ring of partial differential operators (cf. \cite[prop. 3.2, th. 2.1]{Ku4}), then by theorem \ref{schurpair} and proposition \ref{purity} the pair $(CM(A),W)$ also corresponds to a ring of partial differential operators which is Cohen-Macaulay. The corresponding to the pair $(CM(A),W)$ projective surface $X$ is also Cohen-Macalay by \cite[th. 3.2]{Ku4}. 
\end{theo}

\begin{proof} 
Let $X$ be the projective surface corresponding to the pair $(A,W)$ by theorem \ref{dannye}. Then by \cite[th. 3.2]{Ku4} there is a natural isomorphism of a neighbourhood of the divisor $C$ on $X$ and on $CM(X)$ implying $\co_{CM(X),P}\simeq \co_{X,P}$. Thus, we can extend the embedding from definition \ref{mapxi}: $CM(A)\simeq H^0(CM(X)\backslash C, \co_{CM(X)})\hookrightarrow k[[u]]((t))$ (note that the image of this embedding contains $A$). Let's denote the image of this embedding also by $CM(A)$. By the same arguments as in the proof of lemma \ref{claim1} we have $H^0(CM(X), \co_{CM(X)}(nC'))\simeq CM(A)_{nd}$.

Consider the subspace $W'$ in $k[[u]]((t))$ generated by $W$ over  $CM(A)$. Since $W$ is a finitely generated $A$-module, the space $W'$ is generated by finite number elements $w_1,\ldots w_n$ over $CM(A)$ (these elements also generate $W$ over $A$). Because of theorem 3.2 in \cite{Ku4} the graded rings $\gr (CM(A))$ and $\gr (A)$ are equivalent, thus $W'$ is generated as a $k$-subspace by the space $W$ and by finite number of elements $w_ia_j$, where $i=1,\ldots n$, $a_j$ are a basis of finitely dimensional subspace $CM(A)_{kd}$ for some fixed $k$. 

Let $S$ be the operator (see theorem \ref{theo2}) such that $W_0S=\psi_1^{-1}(W)$ (see corollary \ref{trick}). Then we have $B=S\psi_1^{-1}(A)S^{-1}\subset D$ by our assumption, whence $S\in E$ (see the proof of theorem \ref{schurpair} and lemma \ref{lemma8}). Denote by $W_0'$ the space $\psi_1^{-1}(W')S^{-1}$. By the arguments above $W_0'$ is generated by $W_0$ and by finite number of elements $w_ia_jS^{-1}$ as a $k$-space. 
Note that $W_0'B\subset W_0'$ and $W_0'B'\subset W_0'$, where $B'=S\psi_1^{-1}(CM(A))S^{-1}$. 

Now we can argue as in the proof of proposition \ref{prop1} to show that $B'\subset {D}$. Since $S\in E$, we have $B'\in E$. Let $b\in B'$, $b\notin D$. Then $b_-=b-b_+\neq 0$. In this case we have 
$$
0\neq z^{-\ord_{M_1,M_2}(b_-)}b_-= \partial^{\ord_{M_1,M_2}(b_-)}(b_-)(0)\notin W_0
$$
and $z^{-\ord_{M_1,M_2}(b_-)}b_+\in W_0$. Since $W_0'$ is generated by $W_0$ and by finite number of elements not belonging to $W_0$ and since $b\in E$, for some $n\gg 0$ we have  $z^{-\ord_{M_1,M_2}(b_-)-(n,0)}b_-\notin W_0'$. Indeed, let $b_{ij}$ be a coefficient of the series $b_-$ such that $\partial^{\ord_{M_1,M_2}(b_-)}(b_{ij})(0)\neq 0$. Let $b_{i+1,j},\ldots b_{i+q,j}\neq 0$ be non zero coefficients of the series $b_-$ with fixed $j$, i.e. $b_{i+l,j}=0$ for all $l>q$. Then 
for each $n\gg 0$ the condition $z^{-\ord_{M_1,M_2}(b_-)-(n,0)}b_-\in W_0'$ imply the equation 
\begin{multline}
\partial^{\ord_{M_1,M_2}(b_-)}(b_{i,j})(0)+n\partial^{\ord_{M_1,M_2}(b_-)+(1,0)}(b_{i+1,j})(0)+C_n^2\partial^{\ord_{M_1,M_2}(b_-)+(2,0)}(b_{i+2,j})(0)+\ldots\\ +C_n^q\partial^{\ord_{M_1,M_2}(b_-)+(q,0)}(b_{i+q,j})(0)=0.
\end{multline}
Thus for $n=m,\ldots ,m+q+1$ (for $m\gg 0$) a system of linear equations $Cx=0$, $x=(x_0,\ldots ,x_q)$ must hold, where the variables $x_l=\partial^{\ord_{M_1,M_2}(b_-)+(l,0)}(b_{i+l,j})(0)$, $l=0,\ldots q$, and the matrix
$$
C=\left ( \begin{array}{cccc}
1& C_m^1& \ldots & C_m^q\\
1& C_{m+1}^1& \ldots & C_{m+1}^q\\
\vdots & \vdots & \ddots & \vdots\\
1& C_{m+q}^1 &\ldots &C_{m+q}^q\\
\end{array}
\right )
$$
Since $C$ is invertible, we have $x=0$, a contradiction  with $\partial^{\ord_{M_1,M_2}(b_-)}(b_{ij})(0)\neq 0$. So, if $b$ preserves $W_0'$, then $b$ must be in $D$. Therefore, $B'\subset D$ and $B'$ preserves $W_0$. Then $CM(A)$ preserves $W$, hence $(CM(A),W)$ is a Schur pair of rank $r$ (all properties from definition \ref{schurdata}, item \ref{sdat2} for the ring $CM(A)$ hold because $CM(A)\supset A$ is a finite $A$-module). 
\end{proof}

\noindent A. Zheglov,  Lomonosov Moscow State  University, faculty
of mechanics and mathematics, department of differential geometry
and applications, Leninskie gory, GSP, Moscow, \nopagebreak 119899,
Russia
\\ \noindent e-mail
 $azheglov@math.msu.su$

\end{document}